\documentclass[11pt,a4paper]{article}
\usepackage{}
\usepackage{tikz}
\usepackage{amsfonts}
\usepackage{mathrsfs}
\usepackage{multirow}
\usepackage{epsfig}
\usepackage{epstopdf}
\usepackage[T1]{fontenc}
\usepackage{geometry}
\usepackage{amsbsy,amsmath,latexsym,amsfonts, epsfig, color, authblk, amssymb, graphics, bm}
\usepackage{epsf,slidesec,epic,eepic}
\usepackage{fancybox}
\usepackage{fancyhdr}
\usepackage{setspace}
\usepackage{nccmath}
\usepackage{cases}
\usepackage[colorlinks, citecolor=blue]{hyperref}

\newtheorem{theorem}{Theorem}[section]
\newtheorem{lemma}{Lemma}[section]

\newtheorem{definition}{Definition}[section]
\newtheorem{example}{Example}[section]
\newtheorem{proposition}{Proposition}[section]

\newtheorem{remark}{Remark}[section]

\newtheorem{alemma}{Lemma}

\newenvironment{proof}{{\noindent \bf Proof:}}{\hfill$\Box$\medskip}

\definecolor{lred}{rgb}{1,0.8,0.8}
\definecolor{lblue}{rgb}{0.8,0.8,1}
\definecolor{dred}{rgb}{0.6,0,0}
\definecolor{dblue}{rgb}{0,0,0.5}
\definecolor{dgreen}{rgb}{0,0.5,0.5}

 \title{Calmness of partial perturbation to composite rank constraint systems and its applications}

\author{Yitian Qian\footnote{School of Mathematics, South China University of Technology, Guangzhou.},\ \
 Shaohua Pan\footnote{(shhpan@scut.edu.cn) School of Mathematics, South China University of Technology, Guangzhou.}\ \ {\rm and}\ \
 Yulan Liu\footnote{School of Applied Mathematics, Guangdong University of Technology, Guangzhou.}}

\begin{document}

 \maketitle

 \begin{abstract}
  This paper is concerned with the calmness of a partial perturbation to the composite rank
  constraint system, an intersection of the rank constraint set and a general closed set,
  which is shown to be equivalent to a local Lipschitz-type error bound and also a global
  Lipschitz-type error bound under a certain compactness. Based on its lifted formulation,
  we derive two criteria for identifying those closed sets such that the associated partial
  perturbation possesses the calmness, and provide a collection of examples to
  demonstrate that the criteria are satisfied by common nonnegative and positive
  semidefinite rank constraint sets. Then, we use the calmness of this perturbation
  to obtain several global exact penalties for rank constrained optimization
  problems, and a family of equivalent DC surrogates for rank regularized problems.
 \end{abstract}


 \noindent
 {\bf AMS:} 90C31; 54C60; 49K40
 \section{Introduction}\label{sec1}

 Let $\mathbb{R}^{n\times m}$ and $\mathbb{S}^{n}$ respectively denote the space of
 all $n\times m\,(n\le m)$ real matrices and the space of all $n\times n$ real
 symmetric matrices, which are equipped with the trace inner product
 $\langle\cdot,\cdot\rangle$ and its induced Frobenius norm $\|\cdot\|_F$;
 and let $\mathbb{X}$ represent $\mathbb{R}^{n\times m}$ or $\mathbb{S}^{n}$.
 Fix any integer $r\in\{1,2,\ldots,n\}$. Consider the rank constrained optimization problem
 \begin{equation}\label{rank-constr}
  \min_{X\in\Omega}\Big\{f(X)\ \ {\rm s.t.}\ {\rm rank}(X)\le r\Big\}
 \end{equation}
 where $\Omega\subset\mathbb{X}$ is a closed set and $f\!:\mathbb{X}\to(-\infty,+\infty]$
 is a lower bounded function. This model is suitable for the scenario where $r$ is a tight
 upper estimation for the rank of the target matrix. If such $r$ is unavailable,
 one would prefer the rank regularized model
 \begin{equation}\label{rank-reg}
  \min_{X\in\Omega}\Big\{f(X)+\nu{\rm rank}(X)\Big\}
 \end{equation}
 where $\nu>0$ is the regularization parameter. Throughout this paper, we assume that
 the set $\Gamma_{\!r}:=\big\{X\!\in\Omega\,|\,{\rm rank}(X)\le r\big\}$ is nonempty,
 and $f$ is coercive and locally Lipschitz continuous on the set $\Omega$.
 The coercive assumption of $f$ on the set $\Omega$ is very mild and
 guarantees that problem \eqref{rank-constr} has a nonempty solution set. In fact,
 when $f$ does not satisfy this assumption, one may consider replacing $f$ with
 $f+\frac{\mu}{2}\|\cdot\|_F^2$ for a tiny $\mu>0$.

 Models \eqref{rank-constr} and \eqref{rank-reg} have a host of applications
 in statistics \cite{Negahban11}, system identification and control \cite{Fazel02},
 finance \cite{Pietersz04}, machine learning \cite{Kulis09,Hajek16}, and quantum tomography \cite{Gross11}.
 In particular, model \eqref{rank-constr} with $\mathbb{X}=\mathbb{S}^n$ frequently
 arises from the positive semidefinite (PSD) relaxations for combinational and
 graph optimization problems (see, e.g., \cite{Goemans95,Helmberg98,Dukanovic07}).
 Note that ${\rm rank}(X)\le r$ if and only if $\|X\|_*-\|X\|_{(r)}=0$, where $\|\cdot\|_*$
 and $\|\cdot\|_{(r)}$ denote the nuclear norm and the Ky-Fan $r$-norm of matrices, respectively.
 We are interested in the calmness of the following perturbation to the composite rank constraint
 system $\Gamma_{\!r}$:
 \begin{equation}\label{MSr}
  \mathcal{S}_{r}(\tau):=\big\{X\!\in\Omega\,|\, \|X\|_*-\|X\|_{(r)}=\tau\big\}
  \quad{\rm for}\ \tau\in\mathbb{R}.
 \end{equation}

 Clearly, ${\rm dom}\mathcal{S}_{r}:=\{\tau\in\mathbb{R}\,|\,\mathcal{S}_{r}(\tau)\ne\emptyset\}\subseteq\mathbb{R}_{+}$
 because $\|X\|_*\ge\|X\|_{(r)}$ for all $X\in\mathbb{X}$.

 Motivated by the fact that the difference-of-convex (DC) algorithm has been extensively
 explored (see, e.g., \cite{LeThi18,Horst99}), when the set $\Omega$ does not have a simple
 structure (say, the projection onto $\Omega$ has no closed form), it is natural to deal with
 problem \eqref{rank-constr} by penalizing the DC reformulation $\|X\|_*\!-\!\|X\|_{(r)}=0$ of
 the rank constraint, and then develop effective algorithms for solving
 the obtained DC penalized problem
 \begin{equation}\label{epenalty-constr}
  \min_{X\in\Omega}\Big\{f(X)+\rho\big[\|X\|_*\!-\!\|X\|_{(r)}\big]\Big\}
 \end{equation}
 or the factorized form of the penalized problem \eqref{epenalty-constr}, where $\rho>0$ is the penalty parameter. 
 As far as we know, the idea to penalize the DC reformulation of
 the rank constraint first appeared in the technical report \cite{GS-Major}.
 Recently, for the quadratic assignment problem, Jiang et al. \cite{JiangZD21} developed
 a proximal DC approach by the penalized problem \eqref{epenalty-constr} of its equivalent
 rank constrained doubly nonnegative reformulation; and for the unconstrained binary polynomial program,
 Qian and Pan \cite{QianPan21} developed a relaxation approach by the factorized formulation
 of the penalized problem \eqref{epenalty-constr} of its equivalent PSD program.
 The encouraging numerical results in \cite{JiangZD21,QianPan21} inspire us to
 explore the relation between global (or local) optimal solutions of the penalized
 problem \eqref{epenalty-constr} and those of the origin problem \eqref{rank-constr} for
 more closed sets $\Omega$. As will be shown in Section \ref{sec4.1}, the calmness of
 $\mathcal{S}_{r}$ is the key to achieve the goal. This is a motivation
 for us to study the calmness of $\mathcal{S}_{r}$ at $0$.

 Another motivation for studying the calmness of $\mathcal{S}_{r}$ at $0$ is to derive
 equivalent DC surrogates for the rank regularized problem \eqref{rank-reg}.
 It is well known that nonconvex surrogate methods are more effective than
 the nuclear norm convex surrogate method (see, e.g., \cite{Mohan12,BiPan17,Lu14}).
 Take into account that the efficiency of some nonconvex surrogates, such as the Schatten $p$-norm
 \cite{Lai13,Mohan12} and the log-determinant \cite{Fazel03}, depends on their
 approximation level to the rank function. The authors in \cite{LiuBiPan18} derived
 a class of equivalent DC surrogates by the uniformly partial calmness
 of the MPEC reformulation of \eqref{rank-reg}, which includes the matrix version of
 the popular SCAD \cite{Fan01} and MCP \cite{ZhangCH10} surrogates. However,
 the assumption there (see \cite[Theorem 4.2]{LiuBiPan18}) is very restrictive on
 the set $\Omega$ and it may not hold even for a closed ball on the elementwise
 norm of matrices. Then, it is natural to ask if there is a practical criterion
 for identifying more classes of $\Omega$ to obtain such surrogates.

 The last but not least one is to characterize the normal cone to the set $\Gamma_{\!r}$,
 which plays a significant role in deriving the optimality conditions of \eqref{rank-constr}
 (see \cite{LiXiu20}) and verifying the KL property of exponent $1/2$ for its extended objective function.
 Indeed, the two tasks involve the characterization on the normal cone to $\Gamma_{\!r}$.
 By \cite[Section 3.1]{Ioffe08}, the calmness of $\mathcal{S}_{r}$ at $0$ or the
 equivalent metric qualification is enough to achieve an upper inclusion for the normal
 cone to $\Gamma_{\!r}$ in terms of the normal cones to $\Omega$ and the rank constraint set.

 The Aubin property of $\mathcal{S}_{r}$ at $0$ implies its calmness at $0$,
 but one can check by the Mordukhovich criterion \cite{Mordu84} (see also \cite[Theorem 9.40]{RW98})
 that the Aubin property of $\mathcal{S}_{r}$ at $0$ does not hold. In the past few decades,
 there have been a large number of research works on the calmness of a multifunction
 or equivalently the subregularity of its inverse mapping (see, e.g., \cite{BaiYe19,Henrion05,DR09,Gfrerer11,Zheng10})
 and the closely related error bounds of a general lsc function (see, e.g., \cite{Fabian10,Meng12,Ngai09,Kruger15,WuYe03}).
 A collection of criteria have been proposed in these literatures for identifying the calmness of a multifunction,
 but most of them are neighborhood-type and to check if they hold or not
 is not an easy task for a specific $\Omega$.
 One contribution of this work is to present two practical criteria for identifying
 those closed $\Omega$ such that the associated perturbation $\mathcal{S}_{r}$ is
 calm at $0$ for any $X\!\in\mathcal{S}_{r}(0)$; see Section \ref{sec3.1}.
 Although our criteria are stronger than those coming from the above works,
 they are point-type and as will be illustrated in Section \ref{sec3.2},
 they hold for many common nonnegative and PSD composite rank constraint systems.
 Interestingly, the two criteria are precisely the linear regularity of $\Gamma_{r}$
 when it is regarded as an intersection of the set $\Omega$ and the rank constraint set
 or an intersection of a closed set and the positive semidefinite rank constraint set.
 Liner regularity of collections of sets was earliest introduced in \cite{Mordu84} as
 the generalized nonseparation property, and was recently employed in \cite{Lewis09}
 to achieve the local linear convergence for alternating and averaged nonconvex projections.
 For more discussions on the linear regularity of collections of sets, refer to \cite{Kruger06,Ng07}.

 As will be shown in Section \ref{sec3.1}, the calmness of $\mathcal{S}_{r}$ at $0$
 for any $X\in\mathcal{S}_{r}(0)$ is equivalent to a local Lipschitz-type error bound
 and also a global Lipschitz-type error bound under the compactness of $\Omega$.
 To the best of our knowledge, few works discuss the error bounds for rank constrained
 optimization problems except \cite{BiPan16,LiuLu20}. In \cite{BiPan16} the error bound
 was obtained for $\Gamma_{\!r}$ only involving three special $\Omega$ by constructing
 a feasible point technically, while in \cite{LiuLu20} the error bound was established
 only for the spectral norm unit ball $\Omega$. The two papers did not provide a criterion
 to identify the set $\Omega$ such that $\Gamma_{\!r}$ has this property.

 The other contribution of this work is to apply the calmness of $\mathcal{S}_{r}$
 to establishing several classes of global exact penalties for the rank constrained
 problem \eqref{rank-constr}, and deriving a family of equivalent DC surrogates for
 the rank regularized problem \eqref{rank-reg}. For the former, we show that
 the penalized problem \eqref{epenalty-constr}, the Schatten $p$-norm penalized problem,
 and the truncated difference penalized problem of $\|\cdot\|_*$ and $\|\cdot\|_F$
 are all the global exact penalty for problem \eqref{rank-constr},
 which not only generalizes the exact penalty result of \cite{LiuLu20} to more types of $\Omega$,
 but also first verifies the exact penalization for the truncated difference of
 $\|\cdot\|_*$ and $\|\cdot\|_F$ introduced in \cite{MaTH17}. For the latter, we greatly
 improve the result of \cite[Theorem 4.2]{LiuBiPan18} by weakening the restriction
 there on the set $\Omega$; see Section \ref{sec4.2}.

 \section{Notation and preliminaries}\label{sec2}

 Throughout this paper, for each $r\in\{1,2,\ldots,n\}$, we write
 $\Lambda_{r}:=\big\{X\in\mathbb{X}\,|\,{\rm rank}(X)\le r\big\}$ and
 $\Lambda_{r}^{\!+}\!:=\big\{X\!\in\mathbb{S}_{+}^n\,|\,{\rm rank}(X)\le r\big\}$.
 The notation $\mathbb{O}^{n}$ represents the set of all $n\times n$ matrices
 with orthonormal columns, and $I$ and $e$ denote an identity matrix
 and a vector of all ones, respectively, whose dimensions are known from the context.
 For a given $X\in\mathbb{X}$, $\lambda(X)=(\lambda_1(X),\ldots,\lambda_n(X))^{\mathbb{T}}$
 and $\sigma(X)=(\sigma_1(X),\ldots,\sigma_n(X))^{\mathbb{T}}$ denote
 the eigenvalue value and singular value vectors of $X$ arranged in a nonincreasing order.
 For $X\in\mathbb{S}^n$,
 $\mathbb{O}^{n}(X):=\{P\in\mathbb{O}^{n}\,|\, X=P{\rm Diag}(\lambda(X))P^{\mathbb{T}}\}$.
 For a closed set $C\subset\mathbb{X}$, $\Pi_{C}$ denotes the projection mapping
 onto $C$, and for a given $X\in\mathbb{X}$, if $\Pi_C(X)$ is non-unique,
 then $\Pi_C(X)$ denotes any point chosen from this set; ${\rm dist}(X,C)$ means
 the distance from $X$ to the set $C$ in terms of the Frobenius norm; and $\delta_C$
 represents the indicator function of $C$, i.e., $\delta_C(x)=0$ if $x\in C$
 and $+\infty$ otherwise. The notation $\mathbb{B}(\overline{X},\delta)$ denotes
 a closed ball of radius $\delta>0$ centered at $\overline{X}$ with interior
 denoted by $\mathbb{B}^{\circ}(\overline{X},\delta)$, and $\mathbb{B}_{\mathbb{X}}$
 means the unit ball in $\mathbb{X}$. For a linear operator
 $\mathcal{A}\!:\mathbb{X}\to\mathbb{R}^p$, the notation $\mathcal{A}^*$
 denotes its adjoint mapping.
 \subsection{Calmness and subregularity}\label{sec2.1}

  The notion of calmness of a multifunction was first introduced in \cite{Ye97}
  under the term ``pseudo upper-Lipschitz continuity'' owing to the fact that
  it is a combination of Aubin's pseudo-Lipschitz continuity
  and Robinson's upper-Lipschitz continuity \cite{Robinson81}, and the term
 ``calmness'' was later coined in \cite{RW98}. Let $\mathbb{Y}$ and $\mathbb{Z}$
  respectively represent a finite dimensional real vector space equipped with
  the inner product $\langle \cdot,\cdot\rangle$ and its induced norm $\|\cdot\|$.
  A multifunction $\mathcal{M}\!:\mathbb{Y}\rightrightarrows\mathbb{Z}$ is said to
  be calm at $\overline{y}$ for $\overline{z}\in\mathcal{M}(\overline{y})$ if
  there exists a constant $\gamma\ge0$ together with $\varepsilon>0$ and $\delta>0$
  such that for all $y\in\mathbb{B}(\overline{y},\varepsilon)$,
  \begin{equation}\label{calm-def}
   \mathcal{M}(y)\cap\mathbb{B}(\overline{z},\delta)
   \subseteq\mathcal{M}(\overline{y})+\gamma\|y-\overline{y}\|\mathbb{B}_{\mathbb{Z}}.
  \end{equation}
 By \cite[Exercise 3H.4]{DR09}, the neighborhood restriction $\mathbb{B}(\overline{y},\varepsilon)$
 on $y$ in \eqref{calm-def} can be removed. As observed by Henrion and Outrata \cite{Henrion05},
 the calmness of $\mathcal{M}$ at $\overline{y}$ for $\overline{z}\in\mathcal{M}(\overline{y})$
 is equivalent to the (metric) subregularity of its inverse at $\overline{z}$ for
 $\overline{y}\in\mathcal{M}^{-1}(\overline{z})$. Subregularity was introduced
 by Ioffe in \cite{Ioffe79} (under a different name) as a constraint qualification
 related to equality constraints in nonsmooth optimization problems, and was later
 extended to generalized equations. Recall that a multifunction
 $\mathcal{F}\!:\mathbb{Z}\rightrightarrows\mathbb{Y}$ is called (metrically) subregular
 at $\overline{z}$ for $\overline{y}\in\mathcal{F}(\overline{z})$ if there exist
 a constant $\kappa\ge 0$ along with $\varepsilon>0$ such that
 \begin{equation}\label{subregular}
  {\rm dist}(z,\mathcal{F}^{-1}(\overline{y}))\le \kappa{\rm dist}(\overline{y},\mathcal{F}(z))
  \quad{\rm for\ all}\ z\in\mathbb{B}(\overline{z},\varepsilon).
 \end{equation}
 The calmness and subregularity have already been studied by many authors under
 various names (see, e.g., \cite{Henrion02,Henrion05,Ioffe08,Gfrerer11,DR09,Zheng10}
 and the references therein).
\subsection{Normal and tangent cones}\label{sec2.2}

 Let $S\subseteq\mathbb{X}$ be a closed set. The Fr\'{e}chet (regular) normal cone to $S$
 at $\overline{x}\in S$ is defined as
 \[
   \widehat{\mathcal{N}}_{S}(\overline{x})
   :=\bigg\{v\in\mathbb{X}\,|\,\limsup_{\overline{x}\ne x\xrightarrow[S]{}\overline{x}}
    \frac{\langle v,x-\overline{x}\rangle}{\|x-\overline{x}\|}\le 0\bigg\},
 \]
 and the limiting (also called Mordukhovich) normal cone to $S$ at $\overline{x}$ is defined by
 \[
   \mathcal{N}_{S}(\overline{x}):=\Big\{v\in\mathbb{X}\,|\,\exists x^k\xrightarrow[S]{}\overline{x},
   v^k\to v\ {\rm with}\ v^k\in\widehat{\mathcal{N}}_{S}(x^k)\ {\rm for\ all}\ k\Big\}.
 \]
 The tangent (also called Bouligand or contingent)
 cone to $S$ at $\overline{x}$ is defined as
 \[
   \mathcal{T}_{S}(\overline{x}):=\Big\{u\in\mathbb{X}\,|\,\exists u^k\to u,t_k\downarrow 0\ {\rm such\ that}\
   \overline{x}+t_ku^k\in S\ {\rm for\ all}\ k\Big\}.
 \]
 The following lemmas provide the characterization on the normal cone
 to $\mathbb{S}_{+}^n$ and $\Lambda_{r}$.
 \begin{lemma}\label{lemma21}
 (see \cite[Example 2.65]{BS00}) Fix any $X\!\in\mathbb{S}_{+}^n$
  with ${\rm rank}(X)\!=k$. Let $X$ have the eigenvalue decomposition as
 $P{\rm Diag}(\lambda(X))P^{\mathbb{T}}$ with $P\in\mathbb{O}^n$,
 and let $P_1$ be the submatrix consisting of the first $k$ columns of $P$. Then,
 $\mathcal{N}_{\mathbb{S}_{+}^n}(X)=\big\{W\in\mathbb{S}_{-}^n\,|\,P_1^{\mathbb{T}}WP=0\big\}$.
 \end{lemma}
 \begin{lemma}\label{Ncone-MRr}
  (see \cite[Proposition 3.6]{Luke13})
  Fix any $r\in\{1,2,\ldots,r\}$. Consider any $X\!\in\Lambda_r$ with the SVD as
  $U{\rm Diag}(\sigma(X))V^{\mathbb{T}}$ and write $\beta\!:=\big\{i\,|\,\sigma_i(X)=0\big\}$.
  If ${\rm rank}(X)=r$, then
  \(
   \widehat{\mathcal{N}}_{\Lambda_r}(X)
   =\mathcal{N}_{\Lambda_r}(X)=\big\{U_{\!\beta}HV_{\!\beta}^{\mathbb{T}}\,|\, H\in\mathbb{R}^{|\beta|\times|\beta|}\big\};
  \)
 and if ${\rm rank}(X)\!<r$, then it holds that
 $\widehat{\mathcal{N}}_{\Lambda_r}(X)\!=\{0\}\subset\mathcal{N}_{\Lambda_r}(X)
 \!=\!\big\{W\in\mathbb{X}\,|\,{\rm rank}(W)\!\le n\!-r\big\}
 \cap\big\{U_{\beta}HV_{\!\beta}^{\mathbb{T}}\,|\,H\in\mathbb{R}^{|\beta|\times|\beta|}\big\}$.
 \end{lemma}

 By Lemma \ref{lemma21} and \ref{Ncone-MRr}, when $X\!\in\mathbb{S}_{+}^n$,
 it is not hard to verify that $\mathcal{N}_{\mathbb{S}_{+}^n}(X)
 \subseteq\mathcal{N}_{\Lambda_r}(X)$, which implies that
 $\mathcal{N}_{\mathbb{S}_{+}^n}(X)+\mathcal{N}_{\Lambda_r}(X)=\mathcal{N}_{\Lambda_r}(X)$.
 The following lemma provides a characterization on the normal cone to
 the composite set $\Lambda_{r}^{\!+}$.
 \begin{lemma}\label{Ncone-MRr+}
  Fix any $r\in\{1,2,\ldots,n\}$. Consider any point $X\in\Lambda_{r}^{\!+}$. If ${\rm rank}(X)=r$,
  then $\widehat{\mathcal{N}}_{\Lambda_{r}^{\!+}}(X)\!=\mathcal{N}_{\Lambda_{r}^{\!+}}(X)
    =\mathcal{N}_{\Lambda_{r}}(X)$;
  and if ${\rm rank}(X)<r$, then it holds that
  \begin{equation*}
    \mathcal{N}_{\mathbb{S}_{+}^n}(X)\subseteq
   \widehat{\mathcal{N}}_{\Lambda_{r}^{\!+}}(X)\subseteq
   \mathcal{N}_{\Lambda_{r}^{\!+}}(X)\subseteq
   \mathcal{N}_{\mathbb{S}_{+}^n}(X)+\mathcal{N}_{\Lambda_{r}}(X).
  \end{equation*}
 \end{lemma}
 \begin{proof}
  Notice that for any $Y\!\in\mathbb{S}_{+}^n$, ${\rm dist}(Y,\Lambda_{r}^{\!+})={\rm dist}(Y,\Lambda_{r})$.
  Hence, for any $Z\in\mathbb{S}^n$,
  \begin{align}\label{R+set}
  {\rm dist}(Z,\Lambda_{r}^{\!+})&\le\!\|Z\!-\!\Pi_{\mathbb{S}_{+}^n}(Z)\|_F
   +{\rm dist}(\Pi_{\mathbb{S}_{+}^n}(Z),\Lambda_{r}^{\!+})\nonumber\\
  &=\!\|Z\!-\!\Pi_{\mathbb{S}_{+}^n}(Z)\|_F
    +{\rm dist}(\Pi_{\mathbb{S}_{+}^n}(Z),\Lambda_{r})\nonumber\\
  &\le 2{\rm dist}(Z,\mathbb{S}_{+}^n)+{\rm dist}(Z,\Lambda_{r}).
  \end{align}
  Then, by \cite[Section 3.1]{Ioffe08}, $\mathcal{N}_{\Lambda_{r}^{\!+}}(X)
   \subseteq\!\mathcal{N}_{\mathbb{S}_{+}^n}(X)+\!\mathcal{N}_{\Lambda_{r}}(X)$.
  Along with \cite[Theorem 6.42]{RW98},
  \begin{equation}\label{temp-inclusion2}
   \mathcal{N}_{\mathbb{S}_{+}^n}(X)+\widehat{\mathcal{N}}_{\Lambda_r}(X)
   \subseteq\widehat{\mathcal{N}}_{\Lambda_{r}^{\!+}}(X)\subseteq\mathcal{N}_{\Lambda_{r}^{\!+}}(X)
   \subseteq\mathcal{N}_{\mathbb{S}_{+}^n}(X)+\mathcal{N}_{\Lambda_{r}}(X).
  \end{equation}
  When ${\rm rank}(X)=r$, since $\widehat{\mathcal{N}}_{\Lambda_r}(X)=\mathcal{N}_{\Lambda_r}(X)$
  and $\mathcal{N}_{\mathbb{S}_{+}^n}(X)+\mathcal{N}_{\Lambda_r}(X)=\mathcal{N}_{\Lambda_r}(X)$,
  the last inclusions become the desired equalities. When ${\rm rank}(X)<r$,
  since $\widehat{\mathcal{N}}_{\Lambda_r}(X)=\{0\}$,
  we have $\mathcal{N}_{\mathbb{S}_{+}^n}(X)+\widehat{\mathcal{N}}_{\Lambda_r}(X)
  =\mathcal{N}_{\mathbb{S}_{+}^n}(X)$, which by \eqref{temp-inclusion2}
  yields the desired inclusions.
 \end{proof}

 Next we recall from \cite{Gfrerer14,Ye18} the directional version of
 limiting normal cone to a set.
 \begin{definition}\label{dir-normal}
  Given a set $S\subseteq\mathbb{X}$, a point $z\in S$ and a direction
  $d\in\mathbb{X}$, the limiting normal cone to $S$ in direction $d$
  at $z$ is defined by
  \[
   \mathcal{N}_S(z;d):=\!\Big\{v\in\mathbb{X}\!:\,\exists t_k\downarrow 0,d^k\to d,v^k\to v\ {\rm with}\
   v^k\in\widehat{\mathcal{N}}_S(z+t_kd^k)\Big\},
  \]
  and the inner limiting normal cone to $S$ in direction $d$
  at $z$ is defined by
  \[
   \mathcal{N}_S^{i}(z;d):=\!\Big\{v\in\mathbb{Z}\!:\,\forall\,t_k\downarrow 0,d^k\to d,v^k\to v\ {\rm with}\
   v^k\in\widehat{\mathcal{N}}_S(z+t_kd^k)\Big\}.
  \]
 \end{definition}

 From Definition \ref{dir-normal}, it is obvious that $\mathcal{N}_S(z;d)=\emptyset$ if $d\notin \mathcal{T}_S(z)$,
 $\mathcal{N}_S(z;d)\subseteq \mathcal{N}_S(z)$, and $\mathcal{N}_S(z;0)\!=\mathcal{N}_S(z)$.
 When $S$ is convex and $d\in\mathcal{T}_S(z)$,
 $\mathcal{N}_S(z;d)=\mathcal{N}_{\mathcal{T}_S(z)}(d)$.
 \begin{proposition}\label{dnormal-cone}
  Fix any $r\in\{1,2,\ldots,n\}$ and any $X\in\Lambda_r^{\!+}$ with ${\rm rank}(X)=r$. Then,
  \[
    \mathcal{N}_{\Lambda_r^{\!+}}^{i}(X;H)=\mathcal{N}_{\Lambda_r^{\!+}}(X;H)
    =\mathcal{N}_{\Lambda_r^{\!+}}(X)\quad{\rm for\ all}\ H\in\mathcal{T}_{\Lambda_r^{\!+}}(X).
  \]
  In particular, for any $X\in\Lambda_{r}$ with ${\rm rank}(X)=r$, it also holds that
  \[
    \mathcal{N}_{\Lambda_{r}}^{i}(X;H)=\mathcal{N}_{\Lambda_{r}}(X;H)
    =\mathcal{N}_{\Lambda_{r}}(X)\quad{\rm for\ all}\ H\in\mathcal{T}_{\Lambda_{r}}(X).
  \]
 \end{proposition}
 \begin{proof}
  Fix any $0\ne H\in\mathcal{T}_{\Lambda_r^{\!+}}(X)$. Then
  $\mathcal{N}_{\Lambda_r^{\!+}}^{i}(X;H)\subseteq\mathcal{N}_{\Lambda_r^{\!+}}(X;H)
  \subseteq\mathcal{N}_{\Lambda_r^{\!+}}(X)$. So, for the first part, it suffices to prove that
  $\mathcal{N}_{\Lambda_r^{\!+}}(X)\subseteq\mathcal{N}_{\Lambda_r^{\!+}}^{i}(X;H)$.
  Pick any $W\!\in\!\mathcal{N}_{\Lambda_r^{\!+}}(X)$.
  Let $X$ have the eigenvalue decomposition as $P{\rm Diag}(\lambda(X))P^{\mathbb{T}}$
  and write $\beta:=\{i\,|\,\lambda_i(X)=0\}$. By Lemma \ref{Ncone-MRr},
  $W\!=P_{\beta}P_{\beta}^{\mathbb{T}}WP_{\beta}P_{\beta}^{\mathbb{T}}$.
  Since $\Lambda_r^{\!+}$ is Clarke regular at $X$ by Lemma \ref{Ncone-MRr},
  from $H\in\mathcal{T}_{\Lambda_r^{\!+}}(X)$ it follows that for any $\tau_k\downarrow 0$,
  there exists a sequence $\{X^k\}\subseteq\Lambda_{r}^{\!+}$ with $X^k=X+\tau_k(H+\frac{o(\tau_k)}{\tau_k})$.
  Since ${\rm rank}(X)=r$ and $X^k\to X$, there exists
  $\overline{k}\in\mathbb{N}$ such that ${\rm rank}(X^k)\ge r$ for all $k\ge\overline{k}$.
  Along with $\{X^k\}\subseteq\Lambda_r^{\!+}$, we have ${\rm rank}(X^k)=r$
  for each $k\ge\overline{k}$, which implies that $\lambda_i(X^k)\ne 0$
  for $i\notin\beta$ and $\lambda_i(X^k)=0$ for $i\in\beta$
  when $k$ is large enough. For each $k$, let
  $W^k=P_{\beta}^k(P_{\beta}^k)^{\mathbb{T}}WP_{\beta}^k(P_{\beta}^k)^{\mathbb{T}}$
  with $P^k\in\mathbb{O}^{n}(X^k)$. By Lemma \ref{Ncone-MRr},
  we have $W^k\in\widehat{\mathcal{N}}_{\Lambda_r^{\!+}}(X^k)$ for all $k$
  large enough. Since the sequence $\{P^k\}$ is bounded, we may assume that
  (if necessary taking a subsequence) that $P^k\to\widetilde{P}$.
  Clearly, $\widetilde{P}\in\mathbb{O}^{n}(X)$ and
  $W^k\to\widetilde{W}:=\widetilde{P}_{\beta}\widetilde{P}_{\beta}^{\mathbb{T}}
  W\widetilde{P}_{\beta}\widetilde{P}_{\beta}^{\mathbb{T}}$.
  Let $\mu_1>\mu_2>\cdots>\mu_l$ be the distinct eigenvalues of $X$
  and $a_k:=\{i\ |\ \lambda_i(X)=\mu_k\}$ for $k=1,2,\ldots,l$.
  Since $\widetilde{P}\in\mathbb{O}^{n}(X)$ and $P\in\mathbb{O}^{n}(X)$,
  there exists $Q={\rm BlkDiag}(Q_1,\ldots,Q_l)$ with $Q_k\in \mathbb{O}^{|a_k|}$
  for $k=1,2,\ldots,l$ such that $\widetilde{P}=PQ$, which implies that
  $\widetilde{P}_{\beta}\widetilde{P}_{\beta}^{\mathbb{T}}=P_{\beta}P_{\beta}^{\mathbb{T}}$.
  Thus, $\widetilde{W}=W$. By Definition \ref{dir-normal},
  we conclude that $W\in\mathcal{N}_{\Lambda_r^{\!+}}^{i}(X;H)$.
  Using the same arguments, we obtain the second part.
  \end{proof}

  When ${\rm rank}(X)\!<r$, for every $H\in\mathcal{T}_{\Lambda_r^{\!+}}(X)$,
  a tighter upper estimation for $\mathcal{N}_{\Lambda_r^{\!+}}(X;H)$
  than $\mathcal{N}_{\Lambda_r^{\!+}}(X)$ can not be achieved since,
  the exact expression of $\widehat{\mathcal{N}}_{\Lambda_r^{\!+}}(Z)$
  for $Z\in\mathbb{S}^n$ with  ${\rm rank}(Z)<r$ is unavailable.
  Such a difficulty also appears in sparsity constraint sets.
 \section{Calmness of mapping $\mathcal{S}_{r}$ and examples}\label{sec3}

 In this section we establish the calmness of the mapping $\mathcal{S}_{r}$
 under a regularity condition, and illustrate that this condition can be satisfied
 via a collection of common examples.
 \subsection{Calmness of mapping $\mathcal{S}_{r}$}\label{sec3.1}

 First, we achieve the calmness of $\mathcal{S}_{r}$ at $0$ for any
 $\overline{X}\in\mathcal{S}_{r}(0)$ or equivalently a Lipschitz-type local
 error bound for the set $\Gamma_{\!r}$ at any $\overline{X}\in\Gamma_{\!r}$, under a condition coming from
 the metric regularity of a lifted formulation of $\mathcal{S}_{r}$ at $(\overline{X},\overline{X})$ for the origin.
 \begin{theorem}\label{Lcriterion1}
  Consider any $\overline{X}\!\in\Gamma_{\!r}$.
  The mapping $\mathcal{S}_{r}$ is calm at $0$ for $\overline{X}$ if and only if
  either of the following equivalent conditions holds:
  \begin{itemize}
   \item[(i)] there exists a constant $\gamma\ge 0$ along with $\delta>0$ such that
              for all $X\in\mathbb{B}(\overline{X},\delta)$,
              \begin{equation}\label{Lebound1}
                {\rm dist}(X,\Gamma_{\!r})\le\gamma\big[{\rm dist}(X,\Omega)+{\rm dist}(X,\Lambda_{r})\big];
              \end{equation}

  \item[(ii)] $\mathcal{F}_{r}(X,Y)\!:=\!\left\{\!\begin{array}{cl}
               \!\{X\!-\!Y\}\!&\!{\rm if}\ (X,Y)\in\Omega\times\!\Lambda_{r}\!\\
               \emptyset\!& {\rm otherwise}
               \end{array}\right.$ is subregular at $(\overline{X},\overline{X})$ for the origin.
  \end{itemize}
  Consequently, the calmness of $\mathcal{S}_{r}$ at $0$ for any $\overline{X}\in\Gamma_{\!r}$
  is implied by the following condition
  \begin{equation}\label{criterion1}
   [-\mathcal{N}_{\Omega}(\overline{X})]\cap\mathcal{N}_{\Lambda_{r}}(\overline{X})=\{0\}.
  \end{equation}
 \end{theorem}
 \begin{proof}
  By the definition, the calmness of the mapping $\mathcal{S}_{r}$
  at $0$ for $\overline{X}$ is equivalent to the existence of $\gamma'\ge 0$
  and $\delta'>0$ such that for all $Z\in\mathbb{B}(\overline{X},\delta')$,
  \begin{equation}\label{calm-ineq}
    {\rm dist}(Z,\mathcal{S}_{r}(0))\le\gamma'{\rm dist}(0,\mathcal{S}_{r}^{-1}(Z))
    =\left\{\begin{array}{cl}
    \!\gamma'[\|Z\|_*\!-\!\|Z\|_{(r)}]&{\rm if}\ Z\in\Omega;\\
     \infty & {\rm otherwise}
     \end{array}\right.
 \end{equation}
 where the equality is due to the definition of $\mathcal{S}_{r}$ and the fact that
 ${\rm dom}\mathcal{S}_{r}=\mathbb{R}_{+}$.

 \noindent
 {\bf(i)} If there exist $\gamma\ge 0$ and $\delta>0$ such that inequality \eqref{Lebound1}
 holds for all $X\in\mathbb{B}(\overline{X},\delta)$, then inequality \eqref{calm-ineq}
 obviously holds, and the mapping $\mathcal{S}_{r}$ is calm at $0$ for $\overline{X}$.
 Now assume that $\mathcal{S}_{r}$ is calm at $0$ for $\overline{X}$, i.e.,
 there exist $\gamma'\ge 0$ and $\delta'>0$ such that inequality \eqref{calm-ineq}
 holds for all $X\in\mathbb{B}(\overline{X},\delta')$. We will show that
 inequality \eqref{Lebound1} holds with $\delta=\delta'/2$ and $\gamma=1\!+\!\sqrt{n}\gamma'$.
 Pick any $X\in\mathbb{B}(\overline{X},\delta)$. If $X\in\Omega$, by \eqref{calm-ineq}
 the inequality \eqref{Lebound1} holds with $\gamma=\sqrt{n}\gamma'$.
 If $X\notin\Omega$, since $\|\Pi_{\Omega}(X)-\overline{X}\|_F\le 2\|X\!-\!\overline{X}\|_F\le\delta'$,
 from \eqref{calm-ineq} we have
 \begin{align*}
  {\rm dist}(X,\Gamma_{\!r})&\le \|X-\Pi_{\Omega}(X)\|_F+{\rm dist}(\Pi_{\Omega}(X),\mathcal{S}_{r}(0))
  \le {\rm dist}(X,\Omega)+\gamma'{\textstyle\sum_{i=r+1}^n}\sigma_i(\Pi_{\Omega}(X))\\
  &={\rm dist}(X,\Omega)+\gamma'\min_{Z\in\Lambda_{r}}\|Z-\Pi_{\Omega}(X)\|_*\\
  &\le{\rm dist}(X,\Omega)+\gamma'\|X-\Pi_{\Omega}(X)\|_*+\gamma'\min_{Z\in\Lambda_{r}}\|Z-X\|_*\\
  &={\rm dist}(X,\Omega)+\sqrt{n}\gamma'\|X-\Pi_{\Omega}(X)\|_F+\gamma'{\textstyle\sum_{i=r+1}^n}\sigma_i(X)\\
  &\le (1\!+\!\sqrt{n}\gamma'){\rm dist}(X,\Omega)+\gamma'{\textstyle\sum_{i=r+1}^n}\sigma_i(X)
 \end{align*}
 where the equalities are due to Lemma \ref{lemma-SVD} in Appendix.
 So, \eqref{Lebound1} holds with $\gamma=1\!+\!\sqrt{n}\gamma'$.

 \medskip
 \noindent
 {\bf(ii)} It suffices to argue that $\mathcal{F}_{\!r}$ is subregular at
 $(\overline{X},\overline{X})$ for the origin iff part (i) holds.

 \noindent
 $\Longrightarrow$. Since the mapping $\mathcal{F}_{\!r}$ is subregular at $(\overline{X},\overline{X})$
 for the origin, there exists a constant $\kappa\!\ge 0$ along with $\varepsilon\!>0$ such that for all
 $(X,Y)\in\mathbb{B}((\overline{X},\overline{X}),\varepsilon)\cap(\Omega\times\Lambda_{r})$,
 \begin{equation}\label{ineq-MG}
   {\rm dist}((X,Y),\mathcal{F}_{\!r}^{-1}(0))\le\kappa{\rm dist}(0,\mathcal{F}_{\!r}(X,Y)).
  \end{equation}
  Pick any $X\!\in\mathbb{B}(\overline{X},\varepsilon/4)$.
  Obviously, $(X,X)\in\mathbb{B}((\overline{X},\overline{X}),\varepsilon)$.
  If $X\in\Omega\cap\Lambda_{r}=\Gamma_{\!r}$, part (i) automatically holds.
  If $X\in\Omega\backslash\Lambda_{r}$, by noting that $\|\Pi_{\Lambda_{r}}(X)\!-\!\overline{X}\|_F\le2\|X\!-\!\overline{X}\|_F\le\varepsilon/2$,
  we have $(X,\Pi_{\Lambda_{r}}(X))\in\mathbb{B}((\overline{X},\overline{X}),\varepsilon)\cap(\Omega\times\Lambda_{r})$,
  and from \eqref{ineq-MG} it follows that
  \begin{align*}
   {\rm dist}(X,\Gamma_{\!r})&\le{\rm dist}((X,X),\mathcal{F}_{\!r}^{-1}(0))
   \le{\rm dist}((X,\Pi_{\Lambda_{r}}(X)),\mathcal{F}_{\!r}^{-1}(0))+\|\Pi_{\Lambda_{r}}(X)\!-\!X\|_F\\
   &\le\kappa{\rm dist}(0,\mathcal{F}_{\!r}(X,\Pi_{\Lambda_{r}}(X)))+{\rm dist}(X,\Lambda_{r})
   =(1\!+\!\kappa){\rm dist}(X,\Lambda_{r})
  \end{align*}
  where the first inequality is due to $\Gamma_{\!r}\times\Gamma_{\!r}\supseteq\mathcal{F}_{\!r}^{-1}(0)$.
  If $X\in\Lambda_{r}\backslash\Omega$, using the similar arguments yields that
  ${\rm dist}(X,\Gamma_{\!r})\le(1\!+\!\kappa){\rm dist}(X,\Omega)$. Finally, we consider the case that
  $X\notin\Omega\cup\Lambda_{r}$. Since $\|\Pi_{\Omega}(X)\!-\overline{X}\|_F\le\varepsilon/2$
  and $\|\Pi_{\Lambda_{r}}(X)\!-\overline{X}\|_F\le\varepsilon/2$, we have
  $(\Pi_{\Omega}(X),\Pi_{\Lambda_{r}}(X))\in\mathbb{B}((\overline{X},\overline{X}),\varepsilon)\cap(\Omega\times\Lambda_{r})$,
  which along with \eqref{ineq-MG} implies that
  \begin{align*}
   &{\rm dist}(X,\Gamma_{\!r})\le{\rm dist}((X,X),\mathcal{F}_{\!r}^{-1}(0))\\
   &\le{\rm dist}((\Pi_{\Omega}(X),\Pi_{\Lambda_{r}}(X)),\mathcal{F}_{\!r}^{-1}(0))+\|(\Pi_{\Omega}(X),\Pi_{\Lambda_{r}}(X))-(X,X)\|_F\\
   &\le\kappa{\rm dist}(0,\mathcal{F}_{\!r}(\Pi_{\Omega}(X),\Pi_{\Lambda_{r}}(X))+{\rm dist}(X,\Omega)+{\rm dist}(X,\Lambda_{r})\\
   &\le(1\!+\!\kappa)\big[{\rm dist}(X,\Omega)+{\rm dist}(X,\Lambda_{r})\big].
  \end{align*}
  The arguments for the above four cases show that part (i) holds.

 \noindent
 $\Longleftarrow$.
  Since part (i) holds, there exist $\gamma\ge 0$ and $\delta>0$ such that
  inequality \eqref{Lebound1} holds for all $Z\in\mathbb{B}(\overline{X},\delta)$.
  Fix any $(X,Y)\in\mathbb{B}((\overline{X},\overline{X}),\delta)\cap(\Omega\times\Lambda_{r})$.
  Since $X\in\mathbb{B}(\overline{X},\delta)\cap\Omega$ and
  $Y\in\mathbb{B}(\overline{X},\delta)\cap\Lambda_{r}$,
  from inequality \eqref{Lebound1} it immediately follows that
  \[
    \max\{{\rm dist}(X,\Gamma_{\!r}),{\rm dist}(Y,\Gamma_{\!r})\}
    \le\gamma\big[{\rm dist}(Y,\Omega)+{\textstyle\sum_{i=r+1}^n}\sigma_i(X)\big].
  \]
  Observe that ${\rm dist}((X,Y),\mathcal{F}_{\!r}^{-1}(0))
  \le\min_{(Z,Z)\in\Gamma_{\!r}\times\Gamma_{\!r}}\|(Z,Z)-(X,Y)\|_F$.
  Then, we have
  \begin{align*}
   {\rm dist}((X,Y),\mathcal{F}_{\!r}^{-1}(0))
   \le {\rm dist}(X,\Gamma_{\!r})+{\rm dist}(Y,\Gamma_{\!r})
   \le 2\gamma\big[{\rm dist}(Y,\Omega)+{\textstyle\sum_{i=r+1}^n}\sigma_i(X)\big].
  \end{align*}
  This shows that the mapping $\mathcal{F}_{\!r}$ is metrically subregular at
  $(\overline{X},\overline{X})$ for the origin.

  From the equivalence between (i) and (ii), the local error bound in part (i) is implied by
  the metric regularity of $\mathcal{F}_{\!r}$ at $(\overline{X},\overline{X})$ for the origin
  or the Aubin property of its inverse at the origin for $(\overline{X},\overline{X})$.
  The latter is equivalent to $[-\mathcal{N}_{\Omega}(\overline{X})]\cap\mathcal{N}_{\Lambda_{r}}(\overline{X})=\{0\}$.
  Indeed, since ${\rm gph}\mathcal{F}_{\!r}^{-1}\!=\!\mathcal{L}^{-1}(\Omega\times\Lambda_{r}\times\{0\})$
  with $\mathcal{L}(G,X,Y)\!:=(X;Y;G-\!X+Y)$ for $G,X,Y\in\mathbb{X}$, from the surjectivity of
  the mapping $\mathcal{L}$ and \cite[Exercise 6.7 \& Proposition 6.41]{RW98},
  \begin{align}\label{normal}
   &\mathcal{N}_{{\rm gph}\mathcal{F}_{\!r}^{-1}}(0,\overline{X},\overline{X})
    =\mathcal{L}^*[\mathcal{N}_{\Omega}(\overline{X})\times\mathcal{N}_{\Lambda_{r}}(\overline{X})\times\mathbb{X}]\nonumber\\
    &=\!\Big\{(\Delta W,\Delta S\!-\!\Delta W,\Delta Z\!+\!\Delta W)\,|\,\Delta S\in\mathcal{N}_{\Omega}(\overline{X}),
    \Delta Z\in\mathcal{N}_{\Lambda_r}(\overline{X})\Big\}.
  \end{align}
   From \cite[Proposition 3.5]{Mordu94} or \cite[Theorem 9.40]{RW98} it follows that
  the mapping $\mathcal{F}_{\!r}$ has the Aubin property at the origin for $\overline{X}$
  iff $D^*\mathcal{F}_{\!r}((0,0)|\overline{X})(0)=\{(0,0)\}$, or equivalently
  \[
    (\Delta G,0,0)\in\mathcal{N}_{{\rm gph}\mathcal{F}_{\!r}^{-1}}(0,\overline{X},\overline{X})
    \ \Longrightarrow\ \Delta G=0.
  \]
  This, together with \eqref{normal}, is equivalent to saying that $[-\mathcal{N}_{\Omega}(\overline{X})]\cap\mathcal{N}_{\Lambda_{r}}(\overline{X})=\{0\}$.
 \end{proof}

  When $\Omega=\mathbb{S}_{+}^n\cap\Xi$ for a closed set $\Xi\subset\mathbb{S}^n$,
  the condition \eqref{criterion1} does not hold because, by letting $\overline{X}$ have the eigenvalue
  decomposition as $\overline{X}\!=\overline{P}{\rm Diag}(\lambda(\overline{X}))\overline{P}^{\mathbb{T}}$
  and taking $\overline{Z}\!=\!\overline{P}_{\!\beta_1}\overline{P}_{\!\beta_1}^{\mathbb{T}}$
  with $\beta_1\subset\beta\!:=\!\{i\,|\,\lambda_i(\overline{X})=0\}$ for $|\beta_1|\le n\!-r$,
  from Lemma \ref{lemma21} and \ref{Ncone-MRr} we have
  $\overline{Z}\in [-\mathcal{N}_{\mathbb{S}_{+}^n}(\overline{X})]\cap\mathcal{N}_{\Lambda_{r}}(\overline{X})$,
  which together with $\mathcal{N}_{\Omega}(\overline{X})\supseteq
  \widehat{\mathcal{N}}_{\Xi}(\overline{X})+\mathcal{N}_{\mathbb{S}_{+}^n}(\overline{X})$
  means that $0\ne\overline{Z}\in[-\mathcal{N}_{\Omega}(\overline{X})]\cap\mathcal{N}_{\Lambda_r}(\overline{X})$.
  The reason is that the separation of $\mathbb{S}_{+}^n$ from $\Lambda_{r}$
  makes it difficult to hold by recalling that $\mathcal{N}_{\mathbb{S}_{+}^n}(X)\subseteq
  \mathcal{N}_{\Lambda_r}(X)$. Inspired by this, for this class of $\Omega$,
 we achieve the calmness of $\mathcal{S}_{r}$  at $0$ for $\overline{X}\in\Gamma_{\!r}$
 by combining $\mathbb{S}_{+}^n$ and $\Lambda_{r}$.
 \begin{theorem}\label{Lcriterion2}
  Let $\Omega=\mathbb{S}_{+}^n\cap\Xi$ for a closed set $\Xi\subset\mathbb{S}^n$.
  Consider any $\overline{X}\in\Gamma_{\!r}$.
  The mapping $\mathcal{S}_{r}$ is calm at $0$ for $\overline{X}$ under either of
  the equivalent conditions:
  \begin{itemize}
   \item[(i)] there exists a constant $\beta\ge 0$ along with $\varepsilon>0$ such that
              for all $X\in\mathbb{B}(\overline{X},\varepsilon)$,
              \begin{equation}\label{Lebound2}
                {\rm dist}(X,\Gamma_{\!r})\le\beta\big[{\rm dist}(X,\Xi)+{\rm dist}(X,\Lambda_{r}^{\!+})\big];
              \end{equation}

  \item[(ii)] $\mathcal{G}_{r}(X,Y)\!:=\!\left\{\!\begin{array}{cl}
               \!\{X\!-\!Y\}\!&\!{\rm if}\ (X,Y)\in\Xi\times\!\Lambda_{r}^{\!+}\!\\
               \emptyset\!& {\rm otherwise}
               \end{array}\right.$ is subregular at $(\overline{X},\overline{X})$ for the origin;
  \end{itemize}
  and the calmness of $\mathcal{S}_{r}$ at $0$ for $\overline{X}$ is equivalent to either of
  conditions (i) and (ii) if in addition there exists a constant $\kappa'\!>0$ along with $\epsilon'\!>0$ such that
  for all $X\!\in\mathbb{B}(\overline{X},\epsilon')$
  \begin{equation}\label{cond-Xi}
    {\rm dist}(X,\Omega)\le\kappa'\big[{\rm dist}(X,\Xi)+{\rm dist}(X,\mathbb{S}_{+}^n)\big].
  \end{equation}
  Consequently, the calmness of $\mathcal{S}_{r}$ at $0$ for any $\overline{X}\in\Gamma_{\!r}$
  is implied by the following condition
  \begin{equation}\label{criterion2}
   [-\mathcal{N}_{\Xi}(\overline{X})]\cap\mathcal{N}_{\Lambda_{r}^{\!+}}(\overline{X})=\{0\}.
  \end{equation}
 \end{theorem}
 \begin{proof}
  Pick any $X\in\mathbb{B}(\overline{X},\varepsilon)$.
  By combining inequality \eqref{Lebound2} with \eqref{R+set}, it follows that
  \begin{align*}
   {\rm dist}(X,\Gamma_{\!r})
  &\le\beta\big[{\rm dist}(X,\Xi)+2{\rm dist}(X,\mathbb{S}_{+}^n)
    +{\textstyle\sum_{i=r+1}^n}\sigma_i(X)\big]\\
  &\le\beta\big[3{\rm dist}(X,\Omega)+{\textstyle\sum_{i=r+1}^n}\sigma_i(X)\big]
  \end{align*}
  So, part (i) of Theorem \ref{Lcriterion1} holds with $\delta=\varepsilon$
  and $\gamma=3\beta$, and $\mathcal{S}_{r}$ is calm at $0$ for $\overline{X}$.
  By following the same arguments as those for part (ii) of Theorem \ref{Lcriterion1},
  it is not hard to verify that $\mathcal{G}_{r}$ is subregular at
 $(\overline{X},\overline{X})$ for the origin if and only if part (i) holds.

 Next under inequality \eqref{cond-Xi} we argue that the calmness of $\mathcal{S}_{r}$
 at $0$ for $\overline{X}$ implies part (i). Indeed, since $\mathcal{S}_{r}$ is calm at
 $0$ for $\overline{X}$, there exist $\gamma'\ge0$ and $\delta'>0$ such that \eqref{calm-ineq}
 holds for all $X\in\mathbb{B}(\overline{X},\delta')$. Let $\varepsilon=\min(\delta',\varepsilon')$
 and pick any $X\in\mathbb{B}(\overline{X},\varepsilon)$.
 By following the same arguments as those for part (i) of Theorem \ref{Lcriterion1}, we have
 \begin{align*}
  {\rm dist}(X,\Gamma_{\!r})
  &\le (1\!+\!\sqrt{n}\gamma'){\rm dist}(X,\Omega)+\gamma'{\textstyle\sum_{i=r+1}^n}\sigma_i(X)\\
  &\le(1\!+\!\sqrt{n}\gamma')\kappa'\big[{\rm dist}(X,\Xi)+{\rm dist}(X,\mathbb{S}_{+}^n)\big]
      +\gamma'{\textstyle\sum_{i=r+1}^n}\sigma_i(X)\\
  &\le (1\!+\!\sqrt{n}\gamma')\kappa'\big[{\rm dist}(X,\Xi)+{\rm dist}(X,\Lambda_{r}^{\!+})\big]
      +\sqrt{n}\gamma'{\rm dist}(X,\Lambda_{r}^{\!+}).
 \end{align*}
 This means that part (i) holds with $\beta=(1\!+\!\sqrt{n}\gamma')(1\!+\!\kappa')$
 and $\varepsilon=\min(\delta',\varepsilon')$.

 From the equivalence between (i) and (ii), the local error bound in part (i) is implied by
 the metric regularity of $\mathcal{G}_{r}$ at $(\overline{X},\overline{X})$ for the origin
 or the Aubin property of its inverse at the origin for $(\overline{X},\overline{X})$.
 The latter is equivalent to $[-\mathcal{N}_{\Xi}(\overline{X})]\cap\mathcal{N}_{\Lambda_{r}^{\!+}}(\overline{X})=\{0\}$
 by following the similar arguments as those for the last part of Theorem \ref{Lcriterion1}.
 \end{proof}
 \begin{remark}\label{SN-cond}
 {\bf(a)} The condition \eqref{cond-Xi} is equivalent to the calmness at $\overline{X}$ of the mapping
  \[
    \mathcal{M}(X,Y):=\big\{W\in\mathbb{S}^n\,|\,X\!+\!W\in\Xi,Y\!+\!W\in\mathbb{S}_{+}^n\big\}
    \ \ {\rm for}\ X,Y\in\mathbb{S}^n.
  \]
  When the set $\Xi$ is convex, from \cite[Corollary 3]{Bauschke99} the condition
  ${\rm ri}(\Xi)\cap\mathbb{S}_{++}^n\ne\emptyset$ is enough for the condition
  \eqref{cond-Xi} to hold. Clearly, there are many classes of closed convex sets $\Xi$
  to satisfy this constraint qualification. When the set $\Xi$ is nonconvex, by noting that
  the Aubin property of $\mathcal{M}$ at $(\overline{X},\overline{X})$ for the origin is equivalent
  to $[-\mathcal{N}_{\Xi}(\overline{X})]\cap\mathcal{N}_{\mathbb{S}_{+}^n}(\overline{X})=\{0\}$.
  So, in this case, $[-\mathcal{N}_{\Xi}(\overline{X})]\cap\mathcal{N}_{\mathbb{S}_{+}^n}(\overline{X})=\{0\}$
  is enough for the condition \eqref{cond-Xi} to hold.

 \medskip
 \noindent
 {\bf(b)} The conditions \eqref{criterion1} and \eqref{criterion2} are pointed, that is,
 they depends only on the reference point. As will be illustrated in Section \ref{sec3.2},
 by using the characterization on $\mathcal{N}_{\Lambda_{r}}(\overline{X})$ and
 $\mathcal{N}_{\Lambda_{r}^{\!+}}(\overline{X})$, it is convenient to check if
 they hold or not. Although many weaker conditions are available to guarantee
 the calmness of $\mathcal{S}_{r}$ at $0$ for any $\overline{X}\in\mathcal{S}_{r}(0)$
 (see, e.g., \cite{Ngai09,Fabian10,Meng12,Kruger15}), they are all neighborhood-type
 and hard to check in practice. Gfrerer \cite{Gfrerer11} proposed a point-type
 criterion to identify the subregularity of a mapping, but as demonstrated below
 his criterion is only applicable to those $\overline{X}$ with ${\rm rank}(\overline{X})=r$.
 By \cite[Proposition 3.8]{Gfrerer11},
 the mapping $\mathcal{G}_{r}$ is subregular at $(\overline{X},\overline{X})$
 if $(0,0,0)\notin {\rm Cr}_0\mathcal{G}_{r}(\overline{X},\overline{X},0)$ where
  \begin{align*}
   {\rm Cr}_{0}\mathcal{G}_{r}(\overline{X},\overline{X},0)
   &\!:=\!\Big\{(F,S,T)\in\mathbb{S}^n\times\mathbb{S}^n\times\mathbb{S}^n\,|\,\exists
   (G^k,H^k)\in\mathcal{S}_{\mathbb{S}^n\times\mathbb{S}^n},R^k\in\mathcal{S}_{\mathbb{S}^n}
       \nonumber\\
   &\qquad\qquad\qquad (F^k,S^k,T^k)\to (F,S,T),t_k\downarrow 0,\\
   &\qquad\quad (-S^k,-T^k,R^k)\in\widehat{\mathcal{N}}_{{\rm gph}\mathcal{G}_{r}}
   ((\overline{X},\overline{X})\!+t_k(G^k,H^k),t_kF^k)\Big\},
 \end{align*}
 where $\mathcal{S}_{\mathbb{S}^n\times\mathbb{S}^n}$ and $\mathcal{S}_{\mathbb{S}^n}$
 respectively denote the unit sphere in the space $\mathbb{S}^n\times\mathbb{S}^n$ and $\mathbb{S}^n$.
 By Definition \ref{dir-normal}, it is not hard to verify that Gfrerer's criterion is equivalent to
 \begin{equation*}
  \!(0,0,W)\notin \mathcal{N}_{{\rm gph}\mathcal{G}_{r}}((\overline{X},\overline{X},0);(G,H,0))
   \ \ {\rm for\ all}\ ((G,H),W)\in \mathcal{S}_{\mathbb{S}^n\times\mathbb{S}^n}\times \mathcal{S}_{\mathbb{S}^n},
 \end{equation*}
 but unfortunately the exact characterization for the directional normal cone to ${\rm gph}\mathcal{G}_{r}$
 is unavailable. By \cite[Theorem 3.1]{Benko19} and \cite[Proposition 3.3]{Ye18},
 if ${\rm rank}(\overline{X})=r$ or the set $\Xi$ is convex, one may obtain a verifiable
 but stronger version of Gfrerer's criterion
 \begin{equation*}
   (W,-W)\notin \mathcal{N}_{\Xi}(\overline{X};G)\times\mathcal{N}_{\Lambda_{r}^{\!+}}(\overline{X};H)
   \ \ {\rm for\ all}\ W\ne 0,
   (G,H)\in[\mathcal{T}_{\Xi}(\overline{X})\times\mathcal{T}_{\Lambda_{r}^{\!+}}(\overline{X})]\backslash\{(0,0)\}.
 \end{equation*}
 Similarly, by applying Gfrerer's criterion to the mapping $\mathcal{F}_{\!r}$,
 if ${\rm rank}(\overline{X})=r$ or the set $\Omega$ is convex, one may obtain
 a verifiable but stronger version of Gfrerer's criterion
 \begin{equation*}
   (W,-W)\notin \mathcal{N}_{\Omega}(\overline{X};G)\times\mathcal{N}_{\Lambda_{r}}(\overline{X};H)
   \ \ {\rm for\ all}\ W\ne 0,
   (G,H)\in[\mathcal{T}_{\Omega}(\overline{X})\times\mathcal{T}_{\Lambda_{r}}(\overline{X})]\backslash\{(0,0)\}.
 \end{equation*}
 Recall that $\mathcal{N}_{\Xi}(\overline{X};G)\!=\!\mathcal{N}_{\mathcal{T}_{\Xi}(\overline{X})}(G)$
 if $\Xi$ is convex. When ${\rm rank}(\overline{X})=r$ and the set $\Xi$ or $\Omega$ is convex,
 by Proposition \ref{dnormal-cone}, the above two conditions are respectively equivalent to
 \begin{subnumcases}{}
  \label{Dir-ncone-rule1}
  [-\mathcal{N}_{\mathcal{T}_{\Xi}(\overline{X})}(G)]\cap\mathcal{N}_{\Lambda_{r}^{\!+}}(\overline{X})=\{0\}
  \quad{\rm for\ all}\ G\in\mathcal{T}_{\Xi}(\overline{X}),\\
  \label{Dir-ncone-rule2}
   [-\mathcal{N}_{\mathcal{T}_{\Omega}(\overline{X})}(G)]\cap\mathcal{N}_{\Lambda_{r}}(\overline{X})=\{0\}
   \quad{\rm for\ all}\ G\in\mathcal{T}_{\Omega}(\overline{X}).
 \end{subnumcases}
 When $\Xi$ or $\Omega$ is not an affine set, it is possible for \eqref{Dir-ncone-rule1}
 or \eqref{Dir-ncone-rule2} to be weaker than the criterion \eqref{criterion2} or \eqref{criterion1},
 but the former is only applicable to those $\overline{X}$ with ${\rm rank}(\overline{X})=r$.
 \end{remark}

 The following theorem implies that under the compactness of the set $\Omega$,
 the calmness of the mapping $\mathcal{S}_{r}$ at $0$ for all $X\in\mathcal{S}_{r}(0)$
 is equivalent to a global error bound for $\Gamma_{r}$.
 \begin{theorem}\label{gebound}
  Let $\Delta\subseteq\mathbb{X}$ be a compact set. If the mapping $\mathcal{S}_{r}$
  is calm at $0$ for any $X\in\mathcal{S}_{r}(0)$, then there exists a constant $\kappa'>0$
  such that for all $X\in\Delta\cap\Omega$
  \[
    {\rm dist}(X,\Gamma_{\!r})\le\kappa'[\|X\|_*-\|X\|_{(r)}].
  \]
 \end{theorem}
 \begin{proof}
  Since the mapping $\mathcal{S}_{r}$ is calm at $0$ for all $X\in\Gamma_{\!r}$,
 for every $X\!\in\Gamma_{\!r}$ there exist $\kappa_{\!X}\!\ge0$ and $\varepsilon_{\!X}\!>0$ such that
 for all $Z\!\in\Omega\cap\mathbb{B}(X,\varepsilon_{\!X})$,
 \[
   {\rm dist}(Z,\Gamma_{\!r})\le\kappa_{\!X}\big[\|Z\|_*-\|Z\|_{(r)}\big].
 \]
 Since $\bigcup_{X\in\Gamma_{\!r}\cap\Delta}\mathbb{B}^{\circ}(X,\varepsilon_{X})$
 is an open covering of the compact set $\Gamma_{\!r}\cap\Delta$, by Heine-Borel covering theorem,
 there exist $X^1,\ldots,X^p\in\Gamma_{\!r}\cap\Delta$ such that
 $\Gamma_{\!r}\cap\Delta\subseteq\bigcup_{i=1}^p\mathbb{B}^{\circ}(X^i,\varepsilon_{\!X^i})$.
 Write $\widehat{\kappa}:=\max_{1\le i\le p}\kappa_{\!X^i}$.
 From the last inequality, it then follows that
 \[
   {\rm dist}(Z,\Gamma_{\!r})\le\widehat{\kappa}\big[\|Z\|_*-\|Z\|_{(r)}\big]
   \quad{\rm for\ all}\ Z\!\in{\textstyle\bigcup_{i=1}^p}\big[(\Omega\cap\Delta)\cap\mathbb{B}(X^i,\varepsilon_{\!X^i})\big].
 \]
 Let $D=\bigcup_{i=1}^p[(\Omega\cap\Delta)\cap\mathbb{B}(X^i,\varepsilon_{\!X^i})]$.
 Consider the set $\widetilde{\Omega}={\rm cl}[(\Omega\cap\Delta)\backslash D]$.
 Then, there exists $\widetilde{\kappa}>0$ such that
 $\min_{Z\in\widetilde{\Omega}}[\|Z\|_*-\!\|Z\|_{(r)}]\ge\widetilde{\kappa}$.
 If not, there exists a sequence $\{Z^k\}\subseteq\widetilde{\Omega}$ such that
 $[\|Z^k\|_*-\!\|Z^k\|_{(r)}]\le 1/k$, which by the compactness of the set $\widetilde{\Omega}$
 and the continuity of the function $Z\mapsto\|Z\|_*-\!\|Z\|_{(r)}$ means that
 there is a cluster point, say $\overline{Z}\in\widetilde{\Omega}$, of $\{Z^k\}$ such that
 $\|\overline{Z}\|_*-\|\overline{Z}\|_{(r)}=0$. Then
 $\overline{Z}\in\Gamma_{\!r}\cap\Delta\subseteq\bigcup_{i=1}^p\mathbb{B}^{\circ}(X^i,\varepsilon_{\!X^i})$,
 a contradiction to the fact that $\overline{Z}\in\widetilde{\Omega}$. In addition,
 since the sets $\widetilde{\Omega}$ and $\Gamma_{\!r}\cap\Delta$ are compact,
 there exists a constant $c>0$ such that ${\rm dist}(Z,\Gamma_{\!r}\cap\Delta)\le c$
 for all $Z\in\widetilde{\Omega}$. Together with $\min_{Z\in\widetilde{\Omega}}[\|Z\|_*-\!\|Z\|_{(r)}]\ge\widetilde{\kappa}$,
 for any $Z\in\widetilde{\Omega}$ we have
 \(
   {\rm dist}(Z,\Gamma_{\!r}\cap\Delta)\le(c/\widetilde{\kappa})\big[\|Z\|_*\!-\!\|Z\|_{(r)}\big].
 \)
 Consequently, for all $Z\in\widetilde{\Omega}$,
 $ {\rm dist}(Z,\Gamma_{\!r})\le(c/\widetilde{\kappa})\big[\|Z\|_*\!-\!\|Z\|_{(r)}\big]$.
 Along with the last inequality, the desired result holds with
 $\kappa'=\max(\widehat{\kappa},c/\widetilde{\kappa})$. The proof is completed.
 \end{proof}
  \subsection{Some examples}\label{sec3.2}

 In this part we use the criteria \eqref{criterion1} and \eqref{criterion2} to find
 some closed sets $\Omega$ for which the associated mapping $\mathcal{S}_r$ with
 any $r\in\{1,2,\ldots,n\}$ is calm at $0$ for all $\overline{X}\in\mathcal{S}_r(0)$.
 \begin{example}\label{example3.1}
  Fix any $\varrho>0$. Let $\Omega=\!\{Z\in\mathbb{X}\,|\,|\!\|Z|\!\|\le\varrho\}$
  where $|\!\|\cdot|\!\|$ is an arbitrary matrix norm with dual norm $|\!\|\cdot|\!\|_*$.
  Fix any $\overline{X}\in\Gamma_{\!r}$.
  Pick any $H\in[-\mathcal{N}_{\Omega}(\overline{X})]\cap\mathcal{N}_{\Lambda_r}(\overline{X})$.
  From $H\in\mathcal{N}_{\Lambda_r}(\overline{X})$ and Lemma \ref{Ncone-MRr}, we have
  $\langle H,\overline{X}\rangle=0$. From $-H\in\mathcal{N}_{\Omega}(\overline{X})$ and
  the convexity of $\Omega$, for any $Z\in\Omega$ we have
  $0\le\langle H,Z\!-\!\overline{X}\rangle=\langle H,Z\rangle$, which implies that
  \[
    \varrho|\!\|\!-\!H|\!\|_*=\max_{|\!\|Z|\!\|\le\varrho}\langle Z,-H\rangle
    =\max_{Z\in\Omega}\langle Z,-H\rangle\le 0.
  \]
  So, $H=0$ and the criterion \eqref{criterion1} holds at $\overline{X}$.
  When $\Omega=\{Z\in\mathbb{X}\,|\,\|Z\|_F=\varrho\}$, by noting that
  $\mathcal{N}_{\Omega}(\overline{X})=\{\alpha\overline{X}\,|\,\alpha\in\mathbb{R}\}$,
  one can check that the criterion \eqref{criterion1} holds at any $\overline{X}\in\Gamma_{\!r}$.

  Let $\Xi=\!\{Z\in\mathbb{S}^{n}\,|\,|\!\|Z|\!\|\le\varrho\}$.
  Since $\mathcal{N}_{\Lambda_r^{\!+}}(\overline{X})
  \subseteq\mathcal{N}_{\Lambda_r}(\overline{X})+\mathcal{N}_{\mathbb{S}_{+}^n}(\overline{X})$
  by inequality \eqref{R+set}, to verify that the condition \eqref{criterion2} holds,
  we pick any $H\in[-\mathcal{N}_{\Xi}(\overline{X})]\cap(\mathcal{N}_{\Lambda_r}(\overline{X})+\mathcal{N}_{\mathbb{S}_{+}^n}(\overline{X}))$
  and argue that $H=0$. Indeed, since $H\in\mathcal{N}_{\Lambda_r}(\overline{X})+\mathcal{N}_{\mathbb{S}_{+}^n}(\overline{X})$,
  there exist $H^1\in\mathcal{N}_{\Lambda_r}(\overline{X})$ and $H^2\in\mathcal{N}_{\mathbb{S}_{+}^n}(\overline{X})$
  such that $H=H^1+H^2$. By Lemma \ref{lemma21} and \ref{Ncone-MRr}, we have $\langle H^1,\overline{X}\rangle=0$
  and $\langle H^2,\overline{X}\rangle=0$. So, $\langle\overline{X},H\rangle=0$.
  Using the same arguments as above yields that $H=0$.
 \end{example}
 \begin{example}\label{example3.2}
  Let $\Xi=\{X\!\in\mathbb{S}^n\,|\,\mathcal{A}(X)=e\}$ where $\mathcal{A}\!:\mathbb{S}^n\to\mathbb{R}^n$
  is a linear mapping defined by $\mathcal{A}(X)\!:={\rm diag}(X)$.
  Consider any $\overline{X}\in\Gamma_{\!r}$. Since $\mathcal{N}_{\Lambda_r^{\!+}}(\overline{X})
  \subseteq\mathcal{N}_{\Lambda_r}(\overline{X})+\mathcal{N}_{\mathbb{S}_{+}^n}(\overline{X})$,
  to verify that the criterion \eqref{criterion2} holds, we pick any
  $H\!\in{\rm Range}(\mathcal{A}^*)\cap\big[\mathcal{N}_{\mathbb{S}_{+}^n}(\overline{X})
                +\mathcal{N}_{\Lambda_{r}}(\overline{X})\big]$ and argue that $H=0$.
  Clearly, there exist $y\in\mathbb{R}^n,H^1\!\in\mathcal{N}_{\mathbb{S}_{+}^n}(\overline{X})$
  and $H^2\in\mathcal{N}_{\Lambda_{r}}(\overline{X})$ such that
  $H={\rm Diag}(y)\!=H^1+H^2$. From $H^1\!\in\mathcal{N}_{\mathbb{S}_{+}^n}(\overline{X})$,
  we have $H^1\overline{X}=0$. By Lemma \ref{Ncone-MRr}, $H^2\overline{X}=0$.
  Hence, ${\rm Diag}(y)\overline{X}=0$. Along with ${\rm diag}(\overline{X})=e$,
  we obtain $y=0$ and then $H=0$. For the set $\Omega=\Xi\cap\mathbb{S}_{+}^n$,
  the associated $\Gamma_{\!r}$ is the composite rank constraint set in \cite{Pietersz04},
  and when $r=1$, it is the PSD matrix reformulation of the max-cut problem \cite{Goemans95}.
 \end{example}
 \begin{example}\label{example3.3}
  Fix any $b_1\in\mathbb{R}$ and $b_2\in\mathbb{R}\backslash\{0\}$.
  Let $B,C\in\mathbb{S}^n$ be such that $B-(b_1/b_2)C$ is nonsingular. Let
  $\Xi=\!\{X\in\mathbb{S}^n\,|\,\mathcal{A}(X)=b\}$ where $\mathcal{A}\!:\mathbb{S}^n\to\mathbb{R}^2$
  is a linear mapping defined by
  $\mathcal{A}(X)\!:=\left(\begin{matrix}
  \langle B,X\rangle\\ \langle C,X\rangle
  \end{matrix}\right)$.
  Consider any $\overline{X}\!\in\Gamma_{\!r}$. To verify that the criterion \eqref{criterion2}
  holds at $\overline{X}$, we pick any $H\in{\rm Range}(\mathcal{A}^*)
  \cap\big[\mathcal{N}_{\mathbb{S}_{+}^n}(\overline{X})+\mathcal{N}_{\Lambda_r}(\overline{X})\big]$
  and argue that $H=0$. Since ${\rm Range}(\mathcal{A}^*)=\!\{y_1B+y_2C\,|\, y_1\in\mathbb{R},y_2\in\mathbb{R}\}$,
  there exist $y_1\in\mathbb{R}, y_2\in\mathbb{R}, H^1\!\in\mathcal{N}_{\mathbb{S}_{+}^n}(\overline{X})$
  and $H^2\in\mathcal{N}_{\Lambda_{r}}(\overline{X})$ such that
  $H=y_1B+y_2C=H^1+H^2$. Since $\langle \overline{X},H^1\!+\!H^2\rangle=0$,
  we have $y_1b_1 +y_2b_2=0$, which implies that $H=y_1(B-(b_1/b_2)C)$.
  Together with $H^1\overline{X}=0$ and $H^2\overline{X}=0$, we obtain $y_1(B-\!(b_1/b_2)C)\overline{X}=0$,
  which by the assumption on $B$ and $C$ implies $y_1=0$ and then $y_2=0$. Consequently, $H=0$.
  For the set $\Omega=\Xi\cap\mathbb{S}_{+}^n$, when $r=1$, the set $\Gamma_{\!r}$ is exactly
  the feasible set of the generalized eigenvalue problem \cite{Ge16}.

  When $B=0$ and $b_1=0$, $\Xi=\!\{X\in\mathbb{S}^n\,|\,\langle C,X\rangle=b_2\}$ for
  an arbitrary nonsingular $C\in\mathbb{S}^n$. The above arguments show that
  the associated $\mathcal{S}_{r}$ is calm at $0$ for all $\overline{X}\in\Gamma_{\!r}$.
  The set $\Gamma_{\!r}$ with $\Omega=\Xi\cap\mathbb{S}_{+}^n$ for $C=I$
  often appears in quantum state tomography \cite{Gross11}.
 \end{example}
 \begin{example}\label{example3.4}
  Let $\Xi=\!\{X\!\in\mathbb{S}^{n}\,|\, X_{11}\!=\!1,X_{ii}\!-\frac{1}{2}(X_{1i}+\!X_{i1})\!=0\ {\rm for}\ i=2,\ldots,n\}$.
  Consider any $\overline{X}\in\Gamma_{\!r}$. In order verify that the criterion \eqref{criterion2}
  holds at $\overline{X}$, we pick any $H\!\in{\rm Range}(\mathcal{A}^*)
  \cap[\mathcal{N}_{\mathbb{S}_{+}^n}(\overline{X})+\mathcal{N}_{\Lambda_r}(\overline{X})]$
  and argue that $H=0$, where for any $y\in\mathbb{R}^n$
  \[
    \mathcal{A}^*(y)
    =\left[\begin{matrix}
         y_{1} & -\frac{1}{2}y_{2} &-\frac{1}{2}y_{3} & \cdots & -\frac{1}{2}y_{n} \\
         -\frac{1}{2}y_{2} & y_{2} & 0 & \cdots & 0 \\
         -\frac{1}{2}y_{3} & 0 & y_{3} & \cdots & 0 \\
         \vdots & \vdots & \vdots& \ddots & \vdots \\
         -\frac{1}{2}y_{n} & 0 & 0 & \cdots & y_{n} \\
       \end{matrix}\right]\in\mathbb{S}^n.
  \]
  Clearly, there exist $y\in\mathbb{R}^{n},H^1\!\in\mathcal{N}_{\mathbb{S}_{+}^n}(\overline{X})$
  and $H^2\!\in\mathcal{N}_{\Lambda_r}(\overline{X})$ such that
  $H=H^1+H^2=\!\mathcal{A}^*(y)$. Since $H\overline{X}=0$, for $i=2,3,\ldots,n$,
  we have
  $0=(H\overline{X})_{ii}=y_i(\overline{X}_{ii}-\overline{X}_{1i}/2)$ and
  $0=(H\overline{X})_{i1}=y_i(\overline{X}_{i1}-\overline{X}_{11}/2)$.
  Along with $\overline{X}_{11}=1$ and $\overline{X}_{ii}-\overline{X}_{1i}=0$
  for $i=2,\ldots,n$, we obtain $y_i=0$ for all $i=2,\ldots,n$,
  which implies that $0=\langle H,\overline{X}\rangle=y_{1}\overline{X}_{11}=y_1$.
  Consequently, $H=0$. When $r=1$, the set $\Gamma_{\!r}$ associated to
  $\Omega=\mathbb{S}_{+}^n\cap\Xi$ is precisely the PSD matrix reformulation
  for the unconstrained 0-1 quadratic program.
 \end{example}
 \begin{example}\label{example3.5}
  Let $\Xi=\!\big\{X\!\in\mathbb{S}^{n}\,|\,\langle I,X_{ii}\rangle=1,\langle I,X_{ij}\rangle=0,\,i\!\neq\!j\in\{1,\ldots,k\}\big\}$
  with $n=kp$, where $X_{ij}\in\mathbb{S}^p$ is the $(i,j)$th block of $X$.
  Fix any $\overline{X}\in\Gamma_{\!r}$. Let $\mathcal{A}\!:\mathbb{S}^n\to\mathbb{R}^{k^2}$
  be a linear mapping given by $[\mathcal{A}(X)]_{(i-1)k+j}=\langle I,X_{ij}\rangle$
  for $i,j=1,\ldots,k$. To verify that \eqref{criterion2} holds at $\overline{X}$,
  we pick any $H\in{\rm Range}(\mathcal{A}^*)\cap\big[\mathcal{N}_{\mathbb{S}_{+}^n}(\overline{X})
  +\mathcal{N}_{\Lambda_{r}}(\overline{X})\big]$ and argue $H=0$, where
  \[
    \mathcal{A}^*(y)
    =\left[\begin{matrix}
         y_{11}I & y_{12}I & \cdots & y_{1k}I \\
         y_{21}I & y_{22}I & \cdots & y_{2k}I \\
         \vdots & \vdots & \ddots & \vdots \\
         y_{k1}I & y_{k2}I & \cdots & y_{kk}I \\
      \end{matrix}\right]\quad{\rm for}\ y\in\mathbb{R}^{k^2}.
  \]
  Clearly, there exist $y_{ij}\in\mathbb{R}$ for $i,j=1,\ldots,k$,
  $H^1\in\mathcal{N}_{\mathbb{S}_{+}^{n}}(\overline{X})$
  and $H^2\in\mathcal{N}_{\Lambda_r}(\overline{X})$ such that
  $H=\mathcal{A}^*(y)=H^1+H^2$. Multiplying this equality by $\overline{X}$
  yields that $H\overline{X}=H^1\overline{X}+H^2\overline{X}$.
  Notice that $H^1\overline{X}+H^2\overline{X}=0$. So, for all $i,j=1,2,\ldots,k$,
  $0=(H\overline{X})_{ij}=\sum_{t=1}^ky_{it}\overline{X}_{tj}$, and consequently
  $0=\langle I,\sum_{t=1}^ky_{it}\overline{X}_{tj}\rangle=y_{ij}$.
  This means that $H=0$.

  When $r=1$, the set $\Gamma_{\!r}$ associated to $\Omega=\Xi\cap\mathbb{S}_{+}^n$
  is precisely the PSD reformulation for the orthogonal matrix set
  $\{Y\in\mathbb{R}^{p\times k}\ |\ Y^{\mathbb{T}}Y=I\}$
  with $X={\rm vec}(Y){\rm vec}(Y)^{\mathbb{T}}$.
 \end{example}
 \begin{example}\label{example3.6}
  Let $\Omega=\mathbb{R}_{+}^{n\times n}\cap\Delta$ with $\Delta=\{X\!\in\mathbb{R}^{n\times n}\,|\, Xe=e\}$.
  Consider any $\overline{X}\!\in\Gamma_{\!r}$. To verify that the criterion \eqref{criterion1}
  holds, we pick any $H\in[-\mathcal{N}_{\Omega}(\overline{X})]\cap\mathcal{N}_{\Lambda_{r}}(\overline{X})$
  and argue that $H=0$. Since $\mathcal{N}_{\Omega}(\overline{X})
  \!=\mathcal{N}_{\mathbb{R}_{+}^{n\times n}}(\overline{X})+\mathcal{N}_{\Delta}(\overline{X})$,
  there exist $y\in\mathbb{R}^n,H^1\!\in\mathcal{N}_{\mathbb{R}_{+}^{n\times n}}(\overline{X})$
  and $H^2\!\in\mathcal{N}_{\Lambda_{r}}(\overline{X})$ such that
  $H=ye^{\mathbb{T}}\!-\!H^1=H^2$. Note that $\langle H^1,\overline{X}\rangle=0$
  and $\langle H^2,\overline{X}\rangle =0$. Hence,
  $0=\!\langle ye^{\mathbb{T}},\overline{X}\rangle =\langle y,e\rangle$.
  Since $H^2\overline{X}^{\mathbb{T}}=0$ and $-H^1\in\mathbb{R}_{+}^{n\times n}$,
  we have $ny=ye^{\mathbb{T}}e=y(\overline{X}e)^{\mathbb{T}}e=(H^1\!+\!H^2)\overline{X}^{\mathbb{T}}e
  =H^1\overline{X}^{\mathbb{T}}e\le0$,
  which along with $\langle y,e\rangle=0$ implies that $y=0$.
  Together with $\overline{X}^{\mathbb{T}}\!H^2\!=0$,
  it follows that $\overline{X}^{\mathbb{T}}\!H^1=0$. Thus,
  $0=e^{\mathbb{T}}\overline{X}^{\mathbb{T}}H^1=e^{\mathbb{T}}H^1.$
  Combining this with $-H^1\in\mathbb{R}_{+}^{n\times n}$ yields that $H^1=0$.
  So, $H=0$. The set $\Gamma_{\!r}$ associated to such $\Omega$ appears
  in the transition matrix estimation in low-rank Markov chains \cite{LiWZ18}.
 \end{example}
 \begin{example}\label{example3.7}
  Let $\Omega=\mathbb{R}_{+}^{n\times n}\cap\Xi$ with
  $\Delta=\{X\!\in\mathbb{R}^{n\times n}\,|\, Xe=e,X^{\mathbb{T}}e=e\}$.
  Consider any $\overline{X}\in\Gamma_{\!r}$. To verify that the criterion
  \eqref{criterion1} holds, we pick any $H\in[-\mathcal{N}_{\Omega}(\overline{X})]\cap\mathcal{N}_{\Lambda_{r}}(\overline{X})$
  and argue that $H=0$. Since $\mathcal{N}_{\Omega}(\overline{X})
  =\mathcal{N}_{\mathbb{R}_{+}^{n\times n}}(\overline{X})+\mathcal{N}_{\Delta}(\overline{X})$,
  by the expression of $\mathcal{N}_{\Delta}(\overline{X})$, there exist
  $y\in\mathbb{R}^n, z\in\mathbb{R}^n$, $H^1\!\in\mathcal{N}_{\mathbb{R}_{+}^{n\times n}}(\overline{X})$
  and $H^2\in\mathcal{N}_{\Lambda_{r}}(\overline{X})$ such that
  \begin{equation}\label{Hequa0}
     H=ye^{\mathbb{T}}+ez^{\mathbb{T}}-H^1=H^2.
  \end{equation}
  Since $H^2\overline{X}^{\mathbb{T}}\!=0$, we have
  $ye^{\mathbb{T}}\!+ez^{\mathbb{T}}\overline{X}^{\mathbb{T}}\!=H^1\overline{X}^{\mathbb{T}}.$
  Notice that $H^1\in\mathcal{N}_{\mathbb{R}_{+}^{n\times n}}(\overline{X})$. Hence, $(H^1\overline{X}^{\mathbb{T}})_{ii}=0$,
  and $(H^1\overline{X}^{\mathbb{T}})_{ij}\le0$ for all $i,j\in\{1,\ldots,n\}$. Thus,
  \begin{subnumcases}{}
  \label{yplusz0}
   y_i+\overline{X}_{i1}z_1+\ldots+\overline{X}_{in}z_n=0\ \ {\rm for\ all\ } i=1,\ldots,n;\\
   \label{yplusz1}
   y_i+\overline{X}_{j1}z_1+\ldots+\overline{X}_{jn}z_n\le0\ \ {\rm for\ all\ } i,j=1,\ldots,n.
  \end{subnumcases}
  Adding the inequalities in \eqref{yplusz1} from $j=1$ to $n$ yields that
  $ny_i+z_1+\ldots+z_n\le0$ for $i=1,\ldots,n$.
  Notice that $\langle H^1,\overline{X}\rangle=0$
  and $\langle H^2,\overline{X}\rangle =0$. From \eqref{Hequa0},
  we have $0=\langle ye^{\mathbb{T}}+ez^{\mathbb{T}},\overline{X}\rangle
  =\langle y+z,e\rangle$. From the two sides, we have
  $y_i\le\frac{1}{n}\langle y,e\rangle$ for $i=1,\ldots,n$.
  This means that $y_1=\ldots=y_n$ (if not, there is an index $k$
  such that $y_k<\max\limits_{1\le i\le n}y_i$, and then
  $\frac{1}{n}\langle y,e\rangle=\frac{1}{n}(\sum_{i\neq k}y_i+y_k)
  <\max\limits_{1\le i\le n}y_i\le\frac{1}{n}\langle y,e\rangle)$.
  Since $\overline{X}^{\mathbb{T}}\!H^2=0$, by \eqref{Hequa0},
  $\overline{X}^{\mathbb{T}}\!ye^{\mathbb{T}}+ez^{\mathbb{T}}=\!\overline{X}^{\mathbb{T}}H^1.$
  Since $H^1\!\in\mathcal{N}_{\mathbb{R}_{+}^{n\times n}}(\overline{X})$,
  $(\overline{X}^{\mathbb{T}}H^1)_{ij}\le0$ for all $i,j =1,\ldots,n$. Thus,
  $z_i+\overline{X}_{1j}y_1+\ldots+\overline{X}_{nj}y_n\le0$
  for all $i,j=1,\ldots,n$. Adding these inequalities from $j=1$ to $n$ yields
  that $nz_i+y_1+\ldots+y_n\le0$ for $i=1,\ldots,n$.
  Together with $\langle y+z,e\rangle=0$, it implies that
  $z_i\le\frac{1}{n}\langle z,e\rangle$ for all $i=1,\ldots,n.$
  This means that $z_1=z_2=\ldots=z_n$. Combining with $y_1=y_2=\ldots=y_n$ and \eqref{yplusz0},
  we have $y_1+z_1=0$. This implies that $ye^{\mathbb{T}}+ez^{\mathbb{T}}=0$.
  Together with \eqref{Hequa0}, $-H^1=H^2$. Multiplying this equality by
  $\overline{X}^{\mathbb{T}}$ yields that $H^1\overline{X}^{\mathbb{T}}=0$.
  Thus, $0=H^1\overline{X}^{\mathbb{T}}e=H^1e$.
  Note that $H^1\le0$ since $H^1\in\mathcal{N}_{\mathbb{R}_{+}^{n\times n}}(\overline{X})$.
  Consequently, we have $H^1=0$ and $H=0$. The set $\Gamma_{\!r}$ associated to such $\Omega$
  often appears in those problems aiming to seek a low-rank doubly stochastic matrix \cite{Yang16}.
 \end{example}
 \begin{remark}
  {\bf(a)} For Example \ref{example3.1} and \ref{example3.2}, the calmness of the mapping
  $\mathcal{S}_{r}$ at $0$ for any $\overline{X}\in\Gamma_{\!r}$ was shown in \cite{BiPan16}
  by constructing a point in $\Gamma_{\!r}$ technically. Here we achieve it by checking
  the criterion \eqref{criterion1} directly. For Example \ref{example3.3}-\ref{example3.7},
  to the best our knowledge, there is no work to discuss the calmness of the associated
  rank constraint system.

  \medskip
  \noindent
  {\bf(b)} For the above examples, the criterion \eqref{criterion1} or \eqref{criterion2}
  is shown to hold at any $\overline{X}\in\Gamma_{\!r}$. By \cite[Proposition 4.1]{Attouch10},
  for Example \ref{example3.1}, \ref{example3.6} and \ref{example3.7}, the associated
  function $F(X,Y)\!=\frac{1}{2}\|X-\!Y\|_F^2+\delta_{\Omega}(X)+\delta_{\Lambda_{r}}(Y)$
  for $X,Y\in\mathbb{X}$ has the KL property of exponent $1/2$ at $(\overline{X},\overline{X})$;
  while for Example \ref{example3.2}-\ref{example3.5}, the function $G(X,Y)\!=\frac{1}{2}\|X-\!Y\|_F^2+\delta_{\Xi}(X)+\delta_{\Lambda_{r}^{\!+}}(Y)$
  for $X,Y\in\mathbb{S}^n$ has the KL property of exponent $1/2$ at $(\overline{X},\overline{X})$.
  Thus, for these examples, the proximal alternating minimization method \cite{Attouch10}
  can seek a point of $\Gamma_{\!r}$ in a linear rate.

  \medskip
  \noindent
  {\bf(c)} It is easy to verify that the above $\Omega$ except the one in Example \ref{example3.3}
  are all compact. So, the calmness of the associated $\mathcal{S}_{r}$ implies
  the global error bound as in Theorem \ref{gebound}. In addition, since the set $\Omega$
  in the above examples are regular, when $\Omega$ is from Example \ref{example3.1},
  \ref{example3.6} and \ref{example3.7}, for any $X\in\Gamma_{\!r}$ with ${\rm rank}(X)=r$,
  \(
    \widehat{\mathcal{N}}_{\Gamma_{\!r}}(X)=\mathcal{N}_{\Gamma_{\!r}}(X)
    =\mathcal{N}_{\Omega}(X)+\mathcal{N}_{\Lambda_{r}}(X),
  \)
  and for any $X\in\Gamma_{\!r}$ with ${\rm rank}(X)<r$,
  \(
    \mathcal{N}_{\Omega}(X)\subseteq\widehat{\mathcal{N}}_{\Gamma_{\!r}}(X)
    \subseteq\mathcal{N}_{\Gamma_{\!r}}(X)
    \subseteq\mathcal{N}_{\Omega}(X)+\mathcal{N}_{\Lambda_{r}}(X);
  \)
 when $\Omega$ is from Example \ref{example3.2}-\ref{example3.5}, for any $X\in\Gamma_{\!r}$ with ${\rm rank}(X)=r$,
  \(
    \widehat{\mathcal{N}}_{\Gamma_{\!r}}(X)=\mathcal{N}_{\Gamma_{\!r}}(X)
    =\mathcal{N}_{\Xi}(X)+\mathcal{N}_{\Lambda_{r}^{\!+}}(X)
    =\mathcal{N}_{\Xi}(X)+\mathcal{N}_{\mathbb{S}_{+}^n}(X)+\mathcal{N}_{\Lambda_{r}}(X),
  \)
  and for any $X\in\Gamma_{\!r}$ with ${\rm rank}(X)<r$,
  \(
    \mathcal{N}_{\Xi}(X)+\mathcal{N}_{\mathbb{S}_{+}^n}(X)\subseteq\widehat{\mathcal{N}}_{\Gamma_{\!r}}(X)
    \subseteq\mathcal{N}_{\Gamma_{\!r}}(X)
    \subseteq\mathcal{N}_{\Xi}(X)+\mathcal{N}_{\mathbb{S}_{+}^n}(X)+\mathcal{N}_{\Lambda_{r}}(X).
  \)
 \end{remark}

  To close this part, we demonstrate via an example that the calmness of $\mathcal{S}_r$
  associated to the above $\Omega$ can be used to achieve the calmness of
  $\mathcal{S}_r$ with a more complicated $\Omega$.
 \begin{example}\label{example3.8}
  Let $\Xi=\!\{X\!\in\mathbb{S}^{n+1}\,|\,{\rm diag}(X)=e,\mathcal{A}(X)\le b\}$,
  where $b=(b_1,\ldots,b_m)^{\mathbb{T}}$ is a given vector, and the linear mapping
  $\mathcal{A}\!:\mathbb{S}^{n+1}\to\!\mathbb{R}^m$ is defined as follows:
  \[
    \mathcal{A}(X):=(\langle A_1,X\rangle,\ldots,\langle A_m,X\rangle)^{\mathbb{T}}
    \ \ {\rm with}\ \
    A_i=\left(\begin{matrix}
              0 & c_i^{\mathbb{T}}\\
              c_i & Q_i
        \end{matrix}\right)\ {\rm for}\ i=1,\ldots,m.
  \]
  When $r=1$, the set $\Gamma_{\!r}$ associated to $\Omega=\Xi\cap\mathbb{S}_{+}^{n+1}$
  is precisely the feasible set of the PSD matrix reformulation for
  the following binary quadratic programming problem
  \begin{align*}
  &\min_{x\in\{-1,1\}^n}\langle x,Q_0x\rangle+2\langle c_0,x\rangle\nonumber\\
  &\quad\ {\rm s.t.}\ \ \langle x,Q_ix\rangle+2\langle c_i,x\rangle\le b_i,\ \ i=1,2,\ldots,m.
  \end{align*}
  Next we use Example \ref{example3.2} to argue that the mapping $\mathcal{S}_{1}$ associated to
  $\Gamma_{\!1}$ is calm at $0$ for any $\overline{X}\in\Gamma_{\!1}$.
  Write $\widehat{\Gamma}\!:=\{X\in\mathbb{S}_{+}^{n+1}\,|\,{\rm rank}(X)\le 1,{\rm diag}(X)=e\}$.
  Notice that $\widehat{\Gamma}$ is a discrete set. So, there exists $\delta_1>0$ such that
  for all $X\in\mathbb{B}(\overline{X},\delta_1)$, ${\rm dist}(X,\widehat{\Gamma})=\|X\!-\!\overline{X}\|_F$.
  Since $\Gamma_{\!1}\subseteq\widehat{\Gamma}$, $\Gamma_{\!1}$ is also a discrete set in $\mathbb{S}^{n+1}$.
  Then, there exists $\delta_2\in(0,\delta_1]$ such that for all $X\in\mathbb{B}(\overline{X},\delta_2)$,
  ${\rm dist}(X,\Gamma_{\!1})=\|X\!-\!\overline{X}\|_F={\rm dist}(X,\widehat{\Gamma})$.
  Let $\Delta:=\{X\in\mathbb{S}^{n+1}\,|\,{\rm diag}(X)\!=e\}$. By Example \ref{example3.2}
  and Theorem \ref{Lcriterion2} (i), there exist $\beta\ge 0$ and $\delta_3\!>0$ such that
  \[
   {\rm dist}(X,\widehat{\Gamma})\le\beta\big[{\rm dist}(X,\Delta)+{\rm dist}(X,\Lambda_{1}^{\!+})\big]
   \quad{\rm for\ all}\ X\in\mathbb{B}(\overline{X},\delta_3).
  \]
  Take $\delta:=\min(\delta_2,\delta_3)$. From the last inequality
  it follows that for all $X\in\mathbb{B}(\overline{X},\delta)$,
  \[
  {\rm dist}(X,\Gamma_{\!1})={\rm dist}(X,\widehat{\Gamma})
  \le\beta\big[{\rm dist}(X,\Delta)+{\rm dist}(X,\Lambda_{1}^{\!+})\big]
  \le\beta\big[{\rm dist}(X,\Xi)+{\rm dist}(X,\Lambda_{1}^{\!+})\big].
 \]
 By Theorem \ref{Lcriterion2}, the mapping $\mathcal{S}_{1}$ associated to
 $\Omega=\Xi\cap\mathbb{S}_{+}^{n+1}$ is calm at $0$ for any $\overline{X}\in\Gamma_{\!1}$.
 \end{example}
 \section{Applications of calmness of $\mathcal{S}_{r}$}\label{sec4}

 We apply the calmness of $\mathcal{S}_{r}$ at $0$ to achieving several global exact
 penalties for the rank constrained problem \eqref{rank-constr} and a family of equivalent
 DC surrogates for the rank regularized problem \eqref{rank-reg}. Among others,
 the former covers the penalized problem \eqref{epenalty-constr}, the Schatten $p$-norm penalty
 in \cite{LiuLu20} and the truncated difference of $\|\cdot\|_*$ and $\|\cdot\|_F$ in \cite{MaTH17}.
 \subsection{Global exact penalties for problem \eqref{rank-constr}}\label{sec4.1}

 Recall that problem \eqref{rank-constr} is equivalent to the DC constrained optimization problem
 \begin{equation}\label{Erank-constr}
  \min_{X\in\Omega}\Big\{f(X)\ \ {\rm s.t.}\ \ \|X\|_*-\|X\|_{(r)}=0\Big\}.
 \end{equation}
 By \cite[Lemma 2.1\&Proposition 2.1]{LiuBiPan18},
 we have the following global exact penalty result.
 \begin{theorem}\label{theorem1-rconstr}
  If the mapping $\mathcal{S}_{r}$ is calm at $0$ for every $X\!\in\mathcal{S}_{r}(0)$,
  then problem \eqref{Erank-constr} is partially calm at every local optimal solution
  $X^*$, i.e., there exist $\varepsilon\!>0$ and $\overline{\rho}>0$ such that for all
  $\tau\in\mathbb{R}$ and all $X\in\mathbb{B}(X^*,\varepsilon)\cap\mathcal{S}_{r}(\tau)$,
  $f(X)-f(X^*)+\overline{\rho}[\|X\|_*-\|X\|_{(r)}]\geq 0$, and consequently,
  there exists a threshold $\overline{\rho}>0$ such that the problem \eqref{epenalty-constr}
  associated to every $\rho\ge\overline{\rho}$  has the same global optimal
  solution set as the problem \eqref{rank-constr} does.
\end{theorem}
 \begin{remark}\label{remark1-Epenalty}
  {\bf(a)} From Theorem \ref{theorem1-rconstr} and Section \ref{sec3.2}, we conclude that
  the problem \eqref{epenalty-constr} is a global exact penalty for the problem \eqref{rank-constr}
  with $\Omega$ from Example \ref{example3.1}-\ref{example3.7}.

  \medskip
  \noindent
  {\bf(b)} The partial calmness of problem \eqref{Erank-constr} at every local
  optimal solution implies that every local optimal solution of \eqref{rank-constr}
  is locally optimal to the problem \eqref{epenalty-constr} associated to $\rho\ge\overline{\rho}$.
  Conversely, when a local optimal solution of the problem \eqref{epenalty-constr} associated to
  any $\rho>0$ has rank not more than $r$, it must be locally optimal to the problem \eqref{Erank-constr}.
 \end{remark}

 By Theorem \ref{theorem1-rconstr} and Section \ref{sec3.2}, the following theorem
 shows that for the set $\Omega$ from Example \ref{example3.1}-\ref{example3.7},
 the Schatten $p$-norm penalty for problem \eqref{rank-constr} is a global exact one.
 When $\Omega$ is the spectral norm unit ball,
 we recover the exact penalty result in \cite{LiuLu20}.
 \begin{theorem}\label{theorem2-rconstr}
  If the problem \eqref{epenalty-constr} is a global exact penalty for
  the problem \eqref{rank-constr}, then for any $p\in(0,1)$ the following
  problem is also a global exact penalty of \eqref{rank-constr}:
  \begin{equation}\label{epenalty-Lp}
  \min_{X\in\Omega}\Big\{f(X)+\rho\textstyle{\sum_{i=r+1}^n}[\lambda_i(X)]^p\Big\}.
 \end{equation}
 \end{theorem}
 \begin{proof}
  Pick any $p\in(0,1)$. Since the problem \eqref{epenalty-constr} is a global exact penalty
  for \eqref{rank-constr}, there exists $\overline{\rho}>0$ such that the problem \eqref{epenalty-constr}
  associated to every $\rho>\overline{\rho}$ has the same global optimal solution set as the problem
  \eqref{rank-constr} does. For each $\rho>0$, let $\overline{X}_{\!\rho}$ be a global optimal
  solution of problem \eqref{epenalty-Lp} associated to $\rho$.
  Pick a global optimal solution $X^*$ of \eqref{rank-constr}. Then,
  \[
    f(\overline{X}_{\!\rho})+\rho\textstyle{\sum_{i=r+1}^n}[\sigma_i(\overline{X}_{\!\rho})]^p
    \le f(X^*)+\rho\textstyle{\sum_{i=r+1}^n}[\sigma_i(X^*)]^p=f(X^*).
  \]
  Recall that $f$ is lower bounded. There exists a constant $c_0>-\infty$ such that
  $f(\overline{X}_{\!\rho})>c_0$. Together with the last inequality, we have
  $\rho\textstyle{\sum_{i=r+1}^n}[\sigma_i(\overline{X}_{\!\rho})]^p\le f(X^*)-c_0$,
  which implies that there exists $\widehat{\rho}>0$ such that for all $\rho>\widehat{\rho}$,
  $\sigma_i(\overline{X}_{\!\rho})<1$ for $i=r\!+\!1,\ldots,n$. Now fix any
  $\rho>\max(\overline{\rho},\widehat{\rho})$. Let $X_{\!\rho}^*$ be a global
  optimal solution of \eqref{epenalty-constr} associated to $\rho$. Then,
 \begin{align*}
  f(\overline{X}_{\!\rho})+\rho\textstyle{\sum_{i=r+1}^n}[\sigma_i(\overline{X}_{\!\rho})]^p
  &\le f(X_{\!\rho}^*)+\rho\textstyle{\sum_{i=r+1}^n}[\sigma_i(X_{\!\rho}^*)]^p\\
  &=f(X_{\!\rho}^*)+\rho\textstyle{\sum_{i=r+1}^n}\sigma_i(X_{\!\rho}^*)
  \le f(\overline{X}_{\!\rho})+\rho\textstyle{\sum_{i=r+1}^n}\sigma_i(\overline{X}_{\!\rho})
 \end{align*}
 where the first inequality is due to the feasibility of $X_{\!\rho}^*$ to \eqref{epenalty-Lp},
 the second one is by the feasibility of $\overline{X}_{\!\rho}$ to \eqref{epenalty-constr},
 and the equality is using $\sigma_i(X_{\!\rho}^*)=0$ for all $i=r+1,\ldots,n$,
 since $X_{\rho}^*$ is a global optimal solution of \eqref{rank-constr}.
 From the last inequalities, it follows that
 $\rho\textstyle{\sum_{i=r+1}^n}\big[(\sigma_i(\overline{X}_{\!\rho}))^p-\sigma_i(\overline{X}_{\!\rho})\big]\le 0$.
 Together with $0\le\sigma_i(\overline{X}_{\!\rho})<1$ for all $i=r+1,\ldots,n$,
 we deduce that $\sigma_i(\overline{X}_{\!\rho})=0$ for $i=r\!+\!1,\ldots,n$,
 and hence $\overline{X}_{\!\rho}\in\Gamma_{\!r}$. Substituting this into the last inequalities
 yields that $f(X_{\!\rho}^*)=f(\overline{X}_{\!\rho})$. This means that every global optimal
 solution of the problem \eqref{epenalty-Lp} associated to $\rho>\max(\overline{\rho},\widehat{\rho})$
 is globally optimal to \eqref{rank-constr}. In addition, it is easy to argue that every global
 solution of \eqref{rank-constr} is globally optimal to \eqref{epenalty-Lp} associated to any $\rho>0$.
 Thus, the problem \eqref{epenalty-Lp} is a global exact penalty of \eqref{rank-constr}.
\end{proof}

 For each $r\in\{1,2,\ldots,n\}$, let $H_{r}(X)\!:=\sum_{i=1}^{r-1}\sigma_i(X)+\!\sqrt{\sum_{i=r}^n\sigma_i^2(X)}$
 for $X\in\mathbb{X}$. By Lemma \ref{alemma1} in Appendix, problem \eqref{rank-constr}
 is equivalent to the DC constrained problem
 \begin{equation}\label{Erank-constr2}
  \min_{X\in\Omega}\Big\{f(X)\ \ {\rm s.t.}\ \ \|X\|_*-H_{r}(X)=0\Big\},
 \end{equation}
 a truncated difference reformulation of $\|\cdot\|_*$ and $\|\cdot\|_F$ for
 the rank constrained problem. Recently, for the least squared function $f$,
 Ma et al. \cite{MaTH17} studied the penalty problem
 \begin{equation}\label{epenalty-Nnorm}
  \min_{X\in\Omega} f(X)+\rho\big[\|X\|_*-H_{r}(X)\big],
 \end{equation}
 but did not verify its global exactness. By Theorem \ref{theorem1-rconstr} and Section \ref{sec3.2},
 the following theorem shows that it is a global exact penalty of \eqref{rank-constr}
 when $\Omega$ is from Example \ref{example3.1}-\ref{example3.7}.
 \begin{theorem}\label{theorem3-rconstr}
 For each $r\in\{1,2,\ldots,n\}$, let $\mathcal{M}_{r}\!:\mathbb{R}\to\mathbb{X}$ be
 the mapping defined by
 \[
   \mathcal{M}_{r}(\tau):=\big\{X\in\Omega\,|\,\|X\|_*\!-\!H_{r}(X)=\tau\big\}.
 \]
 Fix any $\overline{X}\!\in\Gamma_{\!r}$. Then $\mathcal{S}_{r}$ is calm at $0$
 for $\overline{X}$ iff $\mathcal{M}_{r}$ is calm at $0$ for $\overline{X}$.
 Consequently, when $\Omega$ satisfies the criterion \eqref{criterion1}
 or \eqref{criterion2}, there exists $\overline{\rho}>0$ such that the problem \eqref{epenalty-Nnorm}
 associated to every $\rho\ge\overline{\rho}$ has the same global optimal
 solution set as problem \eqref{rank-constr} does.
 \end{theorem}
 \begin{proof}
  Suppose that $\mathcal{S}_{r}$ is calm at $0$ for $\overline{X}$.
  Then there exist $\delta>0$ and $\kappa\ge0$ such that
  \begin{equation}\label{inclusion1}
    \mathcal{S}_{r}(\tau)\cap\mathbb{B}(\overline{X},\delta)
    \subseteq \mathcal{S}_{r}(0)+\kappa|\tau|
    \quad{\rm for\ all}\ \tau\in\mathbb{R}.
  \end{equation}
  Fix an arbitrary $\tau\in\mathbb{R}$. If $\mathcal{M}_{r}(\tau)=\emptyset$,
  then the following inclusion holds for any $\gamma\ge0$:
  \begin{equation}\label{inclusion2}
   \mathcal{M}_{r}(\tau)\cap\mathbb{B}(\overline{X},\delta)
    \subseteq \mathcal{M}_{r}(0)+\gamma|\tau|.
  \end{equation}
  Now assume that $\mathcal{M}_{r}(\tau)\ne\emptyset$.
  Pick any $X\!\in\mathcal{M}_{r}(\tau)\cap\mathbb{B}(\overline{X},\delta)$.
  Clearly, $\tau=\|X\|_*-H_{r}(X)$. Observe that $X\in\mathcal{S}_{r}(\omega)\cap(\overline{X},\delta)$
  with $\omega=\|X\|_*-\|X\|_{(r)}$. From inclusion \eqref{inclusion1}, we have
  \[
    {\rm dist}(X,\mathcal{M}_{r}(0))={\rm dist}(X,\mathcal{S}_{r}(0))\le \kappa|\omega|
    =\kappa[\|X\|_*\!-\!\|X\|_{(r)}]\le2\kappa[\|X\|_*\!-\!H_{r}(X)]=2\kappa|\tau|
  \]
  where the second inequality is by Lemma \ref{alemma1}. This, along with the arbitrariness
  of $X$ in $\mathcal{M}_{r}(\tau)\cap\mathbb{B}(\overline{X},\delta)$, implies that
  $\mathcal{M}_{r}(\tau)\cap\mathbb{B}(\overline{X},\delta)\subseteq\mathcal{M}_{r}(0)+2\kappa|\tau|$.
  Together with \eqref{inclusion2} and the arbitrariness of $\tau$,
  we conclude that $\mathcal{M}_{r}$ is calm at $0$ for $\overline{X}$.
  Conversely, suppose that the mapping $\mathcal{M}_{r}$ is calm at $0$ for $\overline{X}$.
  Then, there exist $\varepsilon>0$ and $\gamma\ge 0$ such that
  \begin{equation}\label{inclusion11}
   \mathcal{M}_{r}(\tau)\cap\mathbb{B}(\overline{X},\varepsilon)
    \subseteq \mathcal{M}_{r}(0)+\gamma|\tau|
    \quad{\rm for\ all}\ \tau\in\mathbb{R}.
  \end{equation}
  Fix an arbitrary $\tau\in\mathbb{R}$. If $\mathcal{S}_{r}(\tau)=\emptyset$,
  then the following inclusion holds for any $\kappa\ge0$:
  \begin{equation}\label{inclusion21}
   \mathcal{S}_{r}(\tau)\cap\mathbb{B}(\overline{X},\varepsilon)
    \subseteq \mathcal{S}_{r}(0)+\kappa|\tau|.
  \end{equation}
  Now assume that $\mathcal{S}_{r}(\tau)\ne\emptyset$.
  Pick any $X\in\mathcal{S}_{r}(\tau)\cap\mathbb{B}(\overline{X},\varepsilon)$.
  Clearly, $\tau=\|X\|_*-\|X\|_{(r)}$. Note that $X\in\mathcal{M}_{r}(\omega)\cap(\overline{X},\varepsilon)$
  with $\omega=\|X\|_*-H_{r}(X)$. From inclusion \eqref{inclusion11}, we have
  \[
    {\rm dist}(X,\mathcal{S}_{r}(0))={\rm dist}(X,\mathcal{M}_{r}(0))\le \gamma|\omega|
    =\gamma[\|X\|_*\!-\!H_{r}(X)]\le\gamma[\|X\|_*\!-\!\|X\|_{(r)}]=\gamma|\tau|,
  \]
  where the second inequality is due to Lemma \ref{alemma1}. This, by the arbitrariness
  of $X$ in $\mathcal{S}_{r}(\tau)\cap\mathbb{B}(\overline{X},\varepsilon)$, implies that
  \(
    \mathcal{S}_{r}(\tau)\cap\mathbb{B}(\overline{X},\varepsilon)
    \subseteq \mathcal{S}_{r}(0)+\gamma|\tau|.
  \)
  Together with \eqref{inclusion21} and the arbitrariness of $\tau$,
  we conclude that the mapping $\mathcal{S}_{r}$ is calm at $0$ for $\overline{X}$.
 \end{proof}
 \subsection{Equivalent DC surrogates for problem \eqref{rank-reg}}\label{sec4.2}

 Let $\mathscr{L}$ denote the family of proper lsc convex functions $\phi$
 on $\mathbb{R}$ satisfying the conditions:
 \begin{equation*}
   {\rm int}({\rm dom}\,\phi)\supseteq[0,1],\
   t^*\!:=\mathop{\arg\min}_{0\le t\le 1}\phi(t),\ \phi(t^*)=0
   \ \ {\rm and}\ \ \phi(1)=1.
 \end{equation*}
 Many proper lsc convex functions $\phi$ belong to $\mathscr{L}$; see \cite[Appendix]{LiuBiPan18}
 for the examples. For each $\phi\in\mathscr{L}$, let $\psi\!:\mathbb{R}\to(-\infty,+\infty]$ be
 the closed proper convex function given by
 \begin{equation}\label{psifun}
   \psi(t):=\left\{\begin{array}{cl}
                   \phi(t)&{\rm if}\ t\in[0,1];\\
                   +\infty&{\rm otherwise}.
             \end{array}\right.
 \end{equation}
 Pick any $\phi\in\!\mathscr{L}$. It is easy to verify that the problem \eqref{rank-reg}
 is equivalent to the problem
 \begin{equation}\label{MPEC}
  \min_{X\in\Omega,W\in\mathbb{X}}\bigg\{f(X)+\nu\sum_{i=1}^n\phi(\sigma_i(W))
  \ \ {\rm s.t.}\ \ \|X\|_*-\langle W,X\rangle=0,\,\|W\|\le 1\bigg\}
 \end{equation}
 in the sense that if $X^*$ is a global (local) optimal solution of \eqref{rank-reg},
 then $(X^*,W^*)$ with $W^*=U_1^*V_1^*{^{\mathbb{T}}}\!+\!t^*U_2^*V_2^*{^{\mathbb{T}}}$
 is globally (locally) optimal to \eqref{MPEC}, where $U_1^*$ and $V_1^*$ are the matrix
 consisting of the first $r^*={\rm rank}(X^*)$ columns of $U^*$ and $V^*$, and $U_2^*$ and $V_2^*$
 are the matrix consisting of the last $n-r^*$ and $m-r^*$ columns of $U^*$ and $V^*$;
 and if $(X^*,W^*)$ is a global (local) optimal solution of \eqref{MPEC}, then $X^*$ is
 globally (locally) optimal to \eqref{rank-reg}. Let $\mathcal{B}$ denote the spectral norm
 unit ball in $\mathbb{X}$. Notice that $\|X\|_*\!-\!\langle W,X\rangle=0$ and $\|W\|\le 1$
 if and only if $X\!\in\!\mathcal{N}_{\mathcal{B}}(W)$. So, problem \eqref{MPEC} is
 a mathematical program with the equilibrium constraint $X\!\in\!\mathcal{N}_{\mathcal{B}}(W)$.
 In fact, when $\Omega\subseteq\mathbb{S}_{+}^n$, it reduces to
 \begin{equation*}
  \min_{X\in\Omega,W\in\mathbb{S}^n}\bigg\{f(X)+\nu\sum_{i=1}^n\phi(\lambda_i(W))
  \ \ {\rm s.t.}\ \ \langle I-W,X\rangle=0,\,0\preceq W\preceq I\bigg\}.
 \end{equation*}

 The following lemma implies that the MPEC reformulation \eqref{MPEC} for the
 problem \eqref{rank-reg}
 with $\Omega$ from Example \ref{example3.1}-\ref{example3.7} is partially calm
 at every global optimal solution.
 \begin{lemma}\label{partial-calm}
  Suppose that for each $r\in\{1,2,\ldots,n\}$ the mapping $\mathcal{S}_{r}$ is calm at $0$
  for all $X\in\Gamma_{\!r}$, and that $f+\delta_{\Omega}$ is coercive. Then,
  the MPEC problem \eqref{MPEC} is partially calm at every global optimal solution $(X^*,W^*)$,
  i.e., there exist $\delta>0$ and $\overline{\rho}>0$ such that for all $\epsilon\ge 0$
  and all $(X,W)\in\mathbb{B}((X^*,W^*),\delta)\cap\big\{(X,W)\in\Omega\times\mathcal{B}\,|\,
  \|X\|_*-\langle X,W\rangle=\epsilon\big\}$,
  \begin{equation*}
   f(X)+\!\nu\sum_{i=1}^n\phi(\sigma_i(W))
   -\Big[f(X^*)+\!\nu\sum_{i=1}^n\phi(\sigma_i(W^*))\Big]
   +\overline{\rho}\nu\Big(\|X\|_*\!-\langle X,W\rangle\Big)\ge 0.
  \end{equation*}
 \end{lemma}
 \begin{proof}
  Recall that $f$ is assumed to be locally Lipschitz continuous on the set $\Omega$.
  Hence, there exist $\delta'>0$ and $L_{\!f}>0$ such that for any $Z,Z'\in\mathbb{B}(X^*,\delta')\cap\Omega$,
  \begin{equation}\label{fLip}
   |f(Z)-f(Z')|\le L_{\!f}\|Z-Z'\|_F.
  \end{equation}
  Let $\alpha:=\sup_{X\in\Omega\cap\mathbb{B}(X^*,\delta')}f(X)<+\infty$, which is well defined by the continuity of $f$
  on the set $\Omega$. Then, by the coerciveness of $f+\delta_{\Omega}$,
  the set $\mathcal{L}_{\alpha}:=\{X\in\Omega\,|\,f(X)\le\alpha\}$ is compact.
  Since for each $r\in\{1,\ldots,n\}$ the mapping $\mathcal{S}_{r}$ is calm at $0$
  for all $X\in\Gamma_{\!r}$, from Theorem \ref{gebound} it follows that
  for each $r\in\{1,2,\ldots,n\}$ there exists $\kappa_r>0$ such that
  \begin{equation}\label{error-bound1}
   {\rm dist}(Z,\mathcal{S}_{r}(0))\le\kappa_r\big[\|Z\|_*-\|Z\|_{(r)}\big]
   \quad{\rm for\ all}\ Z\in\Omega\cap\mathcal{L}_{\alpha}.
  \end{equation}
  Let $\kappa\!:=\max\{\kappa_1,\ldots,\kappa_n\}$ and $\delta\!:=\!\delta'/2$.
  Take $\overline{\rho}\!:=\!\frac{\kappa\phi_{-}'(1)(1-t^*)L_{\!f}}{\nu(1-t_0)}$ where $t_0\in[0,1)$
  is such that $\frac{1}{1-t^*}\in\partial\phi(t_0)$ and its existence is due to \cite[Lemma 1]{LiuBiPan18}.
  Fix any $\epsilon\ge 0$ and pick any $(X,W)\in\mathbb{B}((X^*,W^*),\delta)\cap
  \big\{(X,W)\in\Omega\!\times\mathcal{B}\,|\,\|X\|_*-\langle X,W\rangle=\epsilon\big\}$.
  Clearly, $X\in\Omega\cap\mathcal{L}_{\alpha}$.
  Define the index set $\overline{J}:=\big\{i\,|\,\overline{\rho}\sigma_i(X)>\phi_{-}'(1)\big\}$
  and write $\overline{r}:=|\overline{J}|$.
  By invoking inequality \eqref{error-bound1}, there necessarily exists a point
  $\overline{X}\in\Pi_{\mathcal{S}_{\overline{r}}(0)}(X)$ such that
  \begin{equation}\label{Xbar-rho}
    \|X-\overline{X}\|_F\le\kappa_{\overline{r}}{\textstyle\sum_{i=\overline{r}+1}^n}\sigma_i(X).
  \end{equation}
  Let $J_1=\big\{i\,|\, \frac{1}{1-t^*}\le\overline{\rho}\sigma_i(X)\le\phi_{-}'(1)\big\}$
  and $J_2=\big\{i\,|\, 0\le\overline{\rho}\sigma_i(X)<\frac{1}{1-t^*}\big\}$.
  Notice that
  \[
   \sum_{i=1}^n\phi(\sigma_i(W))+\overline{\rho}\big(\|X\|_*-\langle W,X\rangle\big)
   \ge\sum_{i=1}^n\min_{t\in[0,1]}\Big\{\phi(t)+\overline{\rho}\sigma_i(X)(1-t)\Big\}.
  \]
  By invoking \cite[Lemma 1]{LiuBiPan18} with $\omega=\sigma_i(X)$,
  we obtain the following inequalities
  \begin{align*}
   &{\textstyle\sum_{i=1}^n}\phi(\sigma_i(W))+\overline{\rho}\big(\|X\|_*-\langle W,X\rangle\big)\\
   &\ge\overline{r}+\frac{\overline{\rho}(1\!-t_0)}{\phi_{-}'(1)(1\!-t^*)}
     \sum_{j\in J_1}\sigma_j(X)+\overline{\rho}(1-t_0)\sum_{j\in J_2}\sigma_j(X)\nonumber\\
   &\ge{\rm rank}(\overline{X})+\frac{\overline{\rho}(1\!-t_0)}{\phi_{-}'(1)(1\!-\!t^*)}
       \sum_{j\in J_1\cup J_2}\!\sigma_j(X)\\
   &\ge {\rm rank}(\overline{X})+ \frac{\overline{\rho}(1\!-t_0)}{\kappa\phi_{-}'(1)(1\!-\!t^*)}\|X\!-\!\overline{X}\|_F\\
   &\ge {\rm rank}(\overline{X})+\nu^{-1}L_{\!f}\|X\!-\!\overline{X}\|_F
    \ge {\rm rank}(\overline{X})+\nu^{-1}\big[f(\overline{X})\!-\!f(X)\big]
  \end{align*}
  where the second inequality is since $\overline{X}\in\Gamma_{\!\overline{r}}$
  and $\phi(t^*)-\phi(1)\ge\phi_{-}'(1)(t^*\!-1)$,
  the third one is due to $J_1\cup J_2=\{j\,|\,\sigma_j(X)<\sigma_{\overline{r}}(X)\}$
  and inequality \eqref{Xbar-rho}, and the last one is using
  $\|\overline{X}-X^*\|_F\le 2\|X-X^*\|_F\le\delta'$ and inequality \eqref{fLip}.
  Now assume that $\overline{X}$ has the SVD given by
  $\overline{U}{\rm Diag}(\sigma(\overline{X}))\overline{V}^{\mathbb{T}}$
  with $\overline{U}\in\mathbb{O}^{n}$ and $\overline{V}\in\mathbb{O}^m$.
  Let $\overline{W}=\overline{U}_1\overline{V}_1^{\mathbb{T}}\!+\!t^*\overline{U}_2\overline{V}_2^{\mathbb{T}}$,
  where $\overline{U}_1$ and $\overline{U}_2$ are the matrix consisting of the first $\overline{r}$
  columns and the rest $n-\overline{r}$ columns of $\overline{U}$, and $\overline{V}_1$
  and $\overline{V}_2$ are the matrix consisting of the first $\overline{r}$
  columns and the rest $m-\overline{r}$ columns of $\overline{V}$.
  Clearly, $(\overline{X},\overline{W})$ is a feasible point of the MPEC \eqref{MPEC} and
  $\sum_{i=1}^n\phi(\sigma_i(\overline{W}))={\rm rank}(\overline{X})$.
  From the last inequality, it immediately follows that
  \begin{align*}
   f(X)+\nu{\textstyle\sum_{i=1}^n}\phi(\sigma_i(W))+\overline{\rho}\nu\big(\|X\|_*\!-\!\langle W,X\rangle\big)
   &\ge f(\overline{X})+\nu{\textstyle\sum_{i=1}^n}\phi(\sigma_i(\overline{W}))\\
   &\ge f(X^*)+\nu{\textstyle\sum_{i=1}^n}\phi(\sigma_i(W^*)).
  \end{align*}
  This gives the desired inequality. The MPEC \eqref{MPEC} is partially calm at $(X^*,W^*)$.
 \end{proof}

 Now we are in a position to provide a family of equivalent DC surrogates for the rank
 regularized problem \eqref{rank-reg}, which greatly improves the result of
 \cite[Corollary 4.2]{BiPan17} where the equivalent surrogates are only achieved
 for the unitarily invariant matrix norm ball.
 \begin{theorem}\label{esurrogate-rankreg}
  If for each $r\in\{1,2,\ldots,n\}$ the mapping $\mathcal{S}_{r}$ is calm
  at $0$ for all $X\in\mathcal{S}_{r}(0)$, then there exists a threshold
  $\overline{\rho}>0$ such that the following problem associated to $\rho\ge\overline{\rho}$
  \begin{equation}\label{Equi-surrogate}
    \min_{X\in\Omega}\bigg\{f(X)+\nu\rho\Big[\|X\|_*-\rho^{-1}\sum_{i=1}^n\!\psi^*(\rho\sigma_i(X))\Big]\bigg\}
  \end{equation}
  has the same global optimal solution set as problem \eqref{rank-reg} does,
  where $\psi^*$ is the conjugate of $\psi$.
  Also, for the set $\Omega$ from Example \ref{example3.1}-\ref{example3.7},
 the assumption automatically holds.
 \end{theorem}
 \begin{proof}
  By combining Lemma \ref{partial-calm} with \cite[Proposition 2.1]{LiuBiPan18}
  and using the expression of $\psi$ in \eqref{psifun},
  there exists a threshold $\overline{\rho}>0$ such that the following penalized problem
  \[
   \min_{X\in\Omega,W\in\mathbb{X}}\Big\{f(X)+\nu\sum_{i=1}^n\psi(\sigma_i(W))
    +\rho\nu\big(\|X\|_*-\langle W,X\rangle\big)\Big\}
  \]
  associated to every $\rho\ge\overline{\rho}$ has the same global optimal solution set
  as the MPEC \eqref{MPEC} does. From the definition of conjugate functions and
  von Neumann's trace inequality, for every $X\in\Omega$ it holds that
  \(
    \sup_{W\in\mathbb{X}}\big\{\langle W,\rho X\rangle-\sum\nolimits_{i=1}^n\!\psi(\sigma_i(W))\big\}
    =\sum\nolimits_{i=1}^n\psi^*(\rho\sigma_i(X)).
  \)
  Consequently, the last penalized problem is simplified as the one in \eqref{Equi-surrogate}.
  The conclusion then follows from the equivalence between \eqref{MPEC} and \eqref{rank-reg}
  in a global sense.
 \end{proof}

 From Theorem \ref{esurrogate-rankreg} and the examples in \cite[Appendix]{LiuBiPan18},
 it follows that for the set $\Omega$ from Example \ref{example3.1}-\ref{example3.7},
 the matrix version of the capped-$\ell_1$, the SCAD and MCP, and the truncated
 $\ell_p\ (0<p<1)$ are all the equivalent DC surrogates for the problem \eqref{rank-reg}.

 \section{Conclusions}\label{sec5}

  For the composite rank constraint system $X\in\Gamma_{\!r}$, we obtained
  two criteria for identifying those closed $\Omega$ such that the associated
  partial perturbation $\mathcal{S}_r$ possesses the calmness at $0$,
  and also illustrated their practicality by a collection of common nonnegative
  and PSD composite rank constraint sets. The calmness of $\mathcal{S}_r$ was
  also used to achieve several global exact penalties for problem \eqref{rank-constr}
  and a family of equivalent DC surrogates for problem \eqref{rank-reg} involving
  more types of $\Omega$. Notice that the results in Section \ref{sec3} are easily
  extended to the mapping $\mathcal{S}_{r}(\tau)\times\Upsilon_{\!s}(\varpi)$, where
  $\Upsilon_{\!s}(\varpi)=\{S\in\Delta\,|\,\|{\rm vec}(S)\|_1-\|{\rm vec}(S)\|_{(r)}=\varpi\}$
  for $\varpi\in\mathbb{R}$ is a partial perturbation to the zero-norm constraint
  $\{S\!\in\Delta\,|\,\|{\rm vec}(S)\|_0\le s\}$. Then, for the rank plus
  zero-norm constrained or regularized problem, one can obtain the corresponding
  global exact penalties and equivalent DC surrogates. Our future work will
  explore other practical criteria to establish error bounds for more
  structured sets.


 \bigskip
 \noindent
 {\bf\large Appendix}
 \begin{alemma}\label{lemma-SVD}
  Let $G\in\!\mathbb{X}$ have the SVD as $U{\rm Diag}(\sigma(G))V^{\mathbb{T}}$.
  Then, for any $r\in\{1,2,\ldots,n\}$,
  \[
    U_1\Sigma_{r}(G)V_1^{\mathbb{T}}\in\mathop{\arg\min}_{Z\in\Lambda_{r}}\|Z-G\|_*,
  \]
  where $U_1$ and $V_1$ are the matrix consisting of the first $r$ columns of $U$ and $V$,
  respectively, and $\Sigma_{r}(G)={\rm diag}(\sigma_1(G),\ldots,\sigma_r(G))$
  with $\sigma_1(G)\ge\sigma_2(G)\ge\cdots\ge\sigma_r(G)$.
 \end{alemma}
 \begin{proof}
  By Mirsky's theorem (see \cite[IV Theorem 4.11]{Stewart90}),
  $\|Z-G\|_*\ge\|\sigma(Z)-\sigma(G)\|_1$ for any $Z\in\mathbb{X}$.
  Then it is easy to argue that if $Z^*$ is an optimal solution of
  $\min_{Z\in\Lambda_r}\|Z-G\|_*$, then $\sigma(Z^*)$ is optimal to
  $\min_{\|z\|_0\le r}\|z-\sigma(G)\|_{1}$. Conversely,
  if $z^*$ is an optimal solution of $\min_{\|z\|_0\le r}\|z-\sigma(G)\|_{1}$,
  then $U{\rm Diag}(|z^*|^{\downarrow})V^{\mathbb{T}}$ is optimal to
  $\min_{Z\in\Lambda_r}\|Z-G\|_*$. Clearly, ${\rm diag}(\Sigma_r(G))$
  is an optimal solution of $\min_{\|z\|_0\le r}\|z-\sigma(G)\|_{1}$.
  The result then holds.
 \end{proof}
 \begin{alemma}\label{alemma1}
  Fix an integer $r\in\{1,2,\ldots,n\}$. Then, for any $X\in\mathbb{X}$,
  it holds that
  \begin{equation}\label{aim-ineq}
     \frac{1}{2}\big[\|X\|_*-\|X\|_{(r)}\big]
     \le{\textstyle\sum_{i=r}^n}\sigma_i(X)-\!\sqrt{{\textstyle\sum_{i=r}^n}\sigma_i^2(X)}
     \le\|X\|_*-\|X\|_{(r)}.
  \end{equation}
 \end{alemma}
 \begin{proof}
  Fix any $X\in\mathbb{X}$. Since $\sqrt{{\textstyle\sum_{i=r}^n}\sigma_i^2(X)}\!\ge\sigma_r(X)$,
  it immediately follows that
  \(
    {\textstyle\sum_{i=r}^n}\sigma_i(X)-\sqrt{{\textstyle\sum_{i=r}^n}\sigma_i^2(X)}
    \le{\textstyle\sum_{i=r+1}^n}\sigma_i(X)=\|X\|_*-\|X\|_{(r)},
  \)
  and the second inequality in \eqref{aim-ineq} holds.
  Next we prove that the first inequality in \eqref{aim-ineq} holds by two cases.

  \noindent
  {\bf Case 1: $\sigma_{r+1}(X)\ne 0$.} Recall that $\sigma_1(X)\ge\sigma_2(X)\ge\cdots\ge\sigma_n(X)$.
  Therefore,
  \begin{align*}
   {\textstyle\sum_{i=r}^n}\sigma_i(X)\!-\!\sqrt{{\textstyle\sum_{i=r}^n}\sigma_i^2(X)}
   &=\sum_{i=r+1}^n\sigma_i(X)+\Big[\sigma_r(X)\!-\!\sqrt{\sigma_r^2(X)+{\textstyle\sum_{i=r+1}^n}\sigma_i^2(X)}\Big]\\
   &=\sum_{i=r+1}^n\sigma_i(X)-\frac{{\textstyle\sum_{i=r+1}^n}\sigma_i^2(X)}{\sigma_r(X)\!+\!\sqrt{\sigma_r^2(X)\!+\!{\textstyle\sum_{i=r+1}^n}\sigma_i^2(X)}}\\
   &
   \ge\sum_{i=r+1}^n\sigma_i(X)-\frac{1}{2}\sum_{i=r+1}^n\frac{\sigma_i(X)}{\sigma_r(X)}\sigma_i(X)\\
   &\ge\sum_{i=r+1}^n\sigma_i(X)-\frac{1}{2}\sum_{i=r+1}^n\sigma_i(X)
   =\frac{1}{2}\big[\|X\|_*-\|X\|_{(r)}\big],
  \end{align*}
  where the last inequality is due to $\sigma_i(X)/\sigma_r(X)\le 1$ for $i=r\!+\!1,\ldots,n$.

  \medskip
  \noindent
  {\bf Case 2: $\sigma_{r+1}(X)=0$.} Now $\|X\|_*-\|X\|_{(r)}=0$
  and ${\textstyle\sum_{i=r}^n}\sigma_i(X)\!-\!\sqrt{{\textstyle\sum_{i=r}^n}\sigma_i^2(X)}=0$.
  Then, it is immediate to have $\frac{1}{2}\big[\|X\|_*-\|X\|_{(r)}\big]
     \le{\textstyle\sum_{i=r}^n}\sigma_i(X)\!-\!\sqrt{{\textstyle\sum_{i=r}^n}\sigma_i^2(X)}$.
 \end{proof}

 \end{document}